\documentclass[11pt,a4paper]{article}
\usepackage{graphicx}
\usepackage[hidelinks]{hyperref}
\usepackage{epsf}
\usepackage{xspace}
\usepackage{tikz}
\usepackage{tkz-euclide}
\usepackage{verbatim}  
\usepackage[ruled,linesnumbered]{algorithm2e}
\newcommand{\nosemic}{\renewcommand{\@endalgocfline}{\relax}}
\newcommand{\dosemic}{\renewcommand{\@endalgocfline}{\algocf@endline}}
\let\oldnl\nl
\newcommand{\nonl}{\renewcommand{\nl}{\let\nl\oldnl}}

\usepackage{a4wide}
\usepackage{amsmath}
\usepackage{amssymb}
\usepackage{amsthm}
\usepackage[utf8]{inputenc}
\usepackage[normalem]{ulem} 
\colorlet{dgreen}{green!60!black}
\colorlet{myblue}{blue!70!black}
\colorlet{myred}{red!70!black}
\definecolor{brown}{rgb}{0.8,0.43,0.2}

\newcommand{\mfn}{\mathfrak{N}}

\newcommand{\mca}{\mathcal{A}}

\newtheorem{de}{Definition} 
 
\newtheorem{lemma}[de]{Lemma} 
\newtheorem{theo}[de]{Theorem} 
\newtheorem{corollary}[de]{Corollary} 
\newtheorem*{remark}{Remark} 
\newtheorem{example}{Example} 
\newcommand{\R}{\mathbb{R}}

\newcommand{\cA}{\mathcal{A}}
\newcommand{\cZ}{\mathcal{Z}}
\newcommand{\cX}{\mathcal{X}}
\newcommand{\act}{{\rm active}}
\newcommand{\con}{{\rm const}}
\newcommand{\conv}{{\rm conv}}
\newcommand{\Name}{barycenter}
\newcommand{\Nam}{Bar}
\DeclareMathOperator{\diam}{diam}
\DeclareMathOperator{\mpd}{mpd}

\begin{document}  
\title{Location problems with cutoff}
\author{Raoul M\"uller\footnote{Work supported by DFG RTG 2088.}\ \footnote{raoul.mueller@uni-goettingen.de}\ \footnote{Institute for Mathematical Stochastics, University of G\"ottingen, 37077 G\"ottingen, Germany.} \and Anita Sch\"obel\footnote{Technische Universität Kaiserslautern, Faculty for Mathematics, and Fraunhofer Institute for Industrial Mathematics ITWM, 67663 Kaiserslautern, Germany.} \and Dominic Schuhmacher\footnotemark[3]\\
}
\maketitle

\begin{abstract}
  In this paper we study a generalized version of the Weber problem of finding a point that minimizes the sum of its distances to a finite number of given points. 
  In our setting these distances may be \emph{cut off} at a given value $C > 0$, and we allow for the option of an \emph{empty} solution at a fixed cost $C'$. 
  We analyze under which circumstances these problems can be reduced to the simpler Weber problem, 
  and also when we definitely have to solve the more complex problem with cutoff.
  
  We furthermore present adaptions of the algorithm of [Drezner et al., 1991, \textit{Transpor\-ta\-tion Science} 25(3), 183--187] to our setting, which in certain situations are able to substantially reduce computation times as demonstrated in a simulation study.
  The sen\-si\-ti\-vi\-ty with respect to the cutoff value is also studied, which allows us to provide an algorithm that efficiently solves the problem simultaneously for all~$C>0$.
\end{abstract}

\parindent 0cm

\section{Introduction}\label{sec:Intro}

For a given finite set $\cA \subseteq\R^k$, a metric $d$ on $\R^k$ and some $q \geq 1$, we study the location problem
\begin{equation} \label{eq:the_problem}
  \min_{z \in \R^k} \sum_{a \in \cA} \min\{d(a,z)^q, C\},
\end{equation}

where $C>0$ is a cutoff parameter. Additionally we allow the
option \emph{not} to choose any location in $\R^k$ at a fixed cost per
point in $\mathcal{A}$.  Without the cutoff $C$ this problem is known
as \emph{Weber} problem, \emph{one-median} problem, \emph{minisum}
problem, \emph{Fermat-Torricelli} problem or (generalized)
\emph{barycenter} problem. In this paper we call an optimal
  solution to the problem a \emph{barycenter}. The barycenter problem
is among the best studied problems in location theory, see
\cite{LaporteNickelSaldanha} for recent
surveys of existing results and new developments in the field. Many
results exist for different metrics and for various
extensions. The problem \eqref{eq:the_problem} introduces the following two extensions to the classic problem.
\medskip

The first extension is the cutoff $C$ (as in \cite{drezner1991facility})
which makes the resulting barycenter more robust against outliers. For $q=1$ and e.g. $d=\ell_1$
this robustness is naturally given, but for other distances, outliers can have a huge effect on the location of the barycenter.
A barycenter can be thought of as a \emph{typical} representative of a given set of points. 
The robustness helps containing this representative property even if outliers are present. 

In the second extension we additionally allow the barycenter to be \emph{empty} at a fixed cost
which is 
constant per point in $\mathcal{A}$. This extends the representative property of the
barycenter. If the given points we want to represent are so scattered
that no single point can represent them, we allow for \emph{no representation}.
\medskip

Many application of the setting are possible where the cutoff $C$ and the possibility of having an empty barycenter come in naturally. An example is a community which has to decide
about building a new waste dump. Anyone can bring their domestic waste for free (but they have the cost of transportation, given by the distance to the waste dump) or have it collected for a fixed cost $C$. If no dump is built, the community has to pay a fixed fee per person to have their waste collected from the waste dump in a nearby city.
\medskip

In the paper we investigate the two extensions and compare their solutions to the solutions of the
classical problem. We identify cases in which solutions to the classical problem are still optimal for the
problem with cutoff and cases in which the empty barycenter is optimal. We also treat the cutoff value $C$ as part of the problem and investigate the sensitivity of an optimal solution w.r.t $C$.
\medskip

Algorithmically, the barycenter problem with cutoff has already been studied, see \cite{drezner1991facility},
\cite{aloise2012improved} and \cite{venkateshan2020note} resulting in an $\mathcal{O}(n^2)$-algorithm for $n$ being the number of existing points in $\mathcal{A}$. 
We refine this algorithm for the two extensions and experimentally show good
computation times.
\medskip

The remainder of the paper is organized as follows:
In the next section we formally introduce the barycenter problem and its two extensions referring to existing literature. 
In Section~\ref{sec:LocalStructure} we look at some universal properties of the cutoff that will be helpful, when we investigate the relation between the barycenter problem \emph{with} and \emph{without} cutoff in Section~\ref{sec:cutoff}. 
Here, we identify cases in which an optimal solution to the classic problem is also optimal for the problem with cutoff. 
Section~\ref{sec:emptyCenter} looks closer at the problem with \emph{empty} barycenter. 
We analyze in which cases the empty barycenter is the best solution. 
In Section~\ref{sec:variableC} we analyze the sensitivity of the barycenter and the objective function value in terms of the cutoff value. 
Section~\ref{sec:applications} sketches an application from statistical data analysis, where the barycenter problem with cutoff and empty set occurs as a subproblem when we compute a ``typical'' point pattern based on a given set of point patterns. 
In Section~\ref{sec:Simulation} we present a simulation study to compare the runtime of the different algorithms. 
The paper ends with some discussions and outlook to further research.

\section{Extensions of the \Name\ location problem: cutoff and empty barycenter}

From now on we will always assume that we are give a finite set of locations $\cA \subseteq\R^k$, $|\mathcal{A}| = n \in \mathbb{N}$. 
The \emph{diameter of $\mathcal{A}$} $$\diam(\mathcal{A}) := \underset{a_1,a_2 \in \mathcal{A}}{\max} d(a_1,a_2)$$ is the maximum distance between two points of $\mathcal{A}$. 
For technical reasons we assume that $\diam(\mathbb{R}^k) = \infty$. 
In this paper we mainly consider \emph{norm-metrics}, i.e., distances 
\[ d^{q}(x,y)= (d(x,y))^q = \|y-x\|^q, \ x,y \in \R^k, q \geq1 \]
\emph{derived from a norm} $\| \cdot \|$ (and here in particular the Euclidean norm
$\ell_2$ and the Manhattan norm $\ell_1$), but many results are also true for general metrics $d$.
When we consider $\ell_p$ norms we allow $p \in [1,\infty]$, so the \emph{maximum norm} is permitted.

\subsubsection*{\eqref{classic}: The barycenter problem}
The classic location problem is to find a point 
$x \in \R^k$ which minimizes the sum of distances to the given points in $\cA$:
\begin{equation*}
\label{classic}
  \mathcal{Z}^\ast := \min_{x \in \R^k} f(x,\cA):=\sum_{a \in \cA} d^{q}(x,a) \tag{\Nam$(\cA)$}. 
\end{equation*}
We call this problem \emph{\Name\ problem} and denote its set of optimal
solutions by $\cX^*$.
If it is clear to which set $\cA$ we refer to we may write $f(x)$ for its objective
function instead of $f(x,\cA)$.

$(\Nam)$ has already been introduced in the 17th century by Fermat for three points $a_1$,$a_2$,
and $a_3$ and for $n$ weighted facilities by Weber in 1909, see, e.g., the survey
\cite{DKSW01}. Actual research concerns versions with $p$ 
facilities \cite{Mladenovich07,Drezetal15,MarinPelegrin20},
barriers \cite{kathrinhabil}, obnoxious facility location \cite{DreDreSch18},
different types of facilities to be placed \cite{MSsurvey,Sch20},
ordered median location problems \cite{Nickel-Puerto,PuertoRodri20}, location under uncertainty \cite{CorreiaSaldanha20}, and others, see \cite{LaporteNickelSaldanha} and references therein for a recent overview. 
Here, we consider the following two extensions of (\Nam).

\subsubsection*{\eqref{cutoff}: The barycenter problem with cutoff}
The first extension we consider is to introduce a cutoff in the distance
function: Given a \emph{cutoff value} $C >0$, we look at the cutoff distance function
\begin{equation}
  d^{q}_{C}(x,y):=\min\{d^{q}(x,y),C\}, \ x,y \in \R^k,
\end{equation}
i.e., the distance is not increased any more once it has reached the
value $C$. The corresponding location problem is given as
\begin{equation}
\label{cutoff}
\mathcal{Z}_C^\ast := \min_{x \in \R^k} f_C(x,\cA):=\sum_{a \in \cA} \min\{d^{q}(x,a), C\}. \tag{\Nam$_C(\cA)$}
\end{equation}
It is called \emph{\Name\ problem with cutoff}. We denote its set of
optimal solutions by $\cX^*_C$. Again, if the set $\mathcal{A}$ is known, we may write $f_C(x)$ instead of $f_C(x,\cA)$. The problem is a special case of the \emph{Weber problem   with limited distance} from \cite{drezner1991facility}.
The latter problem allows different cutoff values $\lambda_i$
for each of the existing facilities while in (\Nam$_C$) all \mbox{$\lambda_i=C$}.
It has also been studied in \cite{aloise2012improved}
and \cite{venkateshan2020note}. Recently, \eqref{cutoff} has been
investigated within a statistical application, namely for finding barycenters for
point patterns, see \cite{muller2020metrics} or Section~\ref{sec:applications}.
At the end of Section~\ref{sec:LocalStructure} we
present the algorithm of \cite{drezner1991facility} for solving the Weber problem with limited distances.

Related work includes \cite{fernandes2017polynomial} where the authors
consider a discrete version of a barycenter problem in which they restrict how many existing points have to be within the cutoff distance. The problem is solved by a global optimization algorithm
based on a decomposition of the plane into regions for which we know which given points
are within the cutoff value $C$. A reversed approach in which one tries to cover as many points as 
possible within a given threshold value $C$ and measures only the distance to the non-covered points is 
investigated in \cite{BJKS13}. 
\medskip

Note that the cutoff does not change the properties of the distances.
Definiteness, symmetry and triangle inequality are still satisfied.
\medskip

\begin{lemma}[\cite{muller2020metrics}]
  If $d$ is a metric then $d_C$ is also a metric. 
\end{lemma}

\subsubsection*{\eqref{empty}: The empty barycenter as an option.}
  
The largest distance to the new facility in \eqref{cutoff} is bounded by the cutoff value $C$. 
In the second extension we go a step further and allow to place
no facility (represented as $x=\emptyset$). In this case, each demand point
$a \in \cA$ has to pay a price of $C':=\alpha \cdot C$ for some given $\alpha>0$.
In order to formulate this setting as location problem, we extend the metric space by the empty set $\emptyset$ for which we define
a constant ``distance'' between $x=\emptyset$ and any other
point $y\in \mathbb{R}^k \cup \emptyset$, namely
\[ d^q_{C,\alpha}(\emptyset,y)=\left\{ \begin{array}{ll}
                              \alpha \cdot C & \mbox{ if } y \not= \emptyset\\
                              0  & \mbox{ if } y = \emptyset
                                       \end{array} \right. \]
and leave $d^q_{C,\alpha}(x,y)=d^q_{C}(x,y)$ for all $x,y \not= \emptyset$.
The corresponding location problem
\begin{equation}
\label{empty}
\mathcal{Z}_{C,\alpha}^\ast := \min_{x \in \R^k \cup \{\emptyset\}} f_{C,\alpha}(x,\cA):=\sum_{a \in \cA} d^{q}_{C,\alpha}(x,a),
\tag{\Nam$_{C,\alpha}(\cA)$}
\end{equation}
is called \emph{\Name\ problem with empty set (and cutoff)}. We denote its set of
optimal solutions by $\cX^*_{C,\alpha}$ and call $\xi^\ast = \emptyset$ the \emph{empty barycenter}. 
The problem has recently been introduced and motivated in
\cite{muller2020metrics} but to the best of our knowledge otherwises not been studied.
\medskip

Adding the empty barycenter to the metric space with cutoff distance $d_C$
  does not change the properties of the metric space if $\alpha \geq \frac{1}{2}$.
\medskip

\begin{lemma}\label{lem:MetricSpaceWithEmptyset}
  $M^\prime = (\mathbb{R}^k \cup \emptyset, d_{C,\alpha})$ is a metric space if and only if $\alpha \geq \frac{1}{2}$.
\end{lemma}

\begin{proof}
  The definiteness and the symmetry of the metric $d_{C,\alpha}$ directly hold
  also for $\emptyset$. The triangle inequality
  \begin{equation}
  \label{triangle}  
    d_{C,\alpha}(x,y) + d_{C,\alpha}(y,z) \geq d_{C,\alpha}(x,z)
  \end{equation}
  can be shown by checking all possible cases:
  \begin{itemize}
    \item If $x,y,z \in M$, \eqref{triangle} is satisfied since $d_C$ is a metric.
    \item If $x=y=z=\emptyset$, or if exactly two of the three points $x,y,z$ are $\emptyset$, \eqref{triangle} follows directly from the definition of $d_{C,\alpha}$.
    \item For only $x=\emptyset$ the triangle inequality holds since $d_{C,\alpha}(y,z) \geq 0$.
      The same holds for $z=\emptyset$.
\end{itemize}
We are left with the case that $y=\emptyset$ and $x,z \in M$. In this case, \eqref{triangle}
transfers to 
\begin{equation}
\label{eq-triangle}  
  \alpha C+ \alpha C\geq d_{C,\alpha}(x,z) = d_C(x,z).
\end{equation}
We have to show two directions:
\begin{description}
\item{$\Longrightarrow$}
  Let \eqref{eq-triangle} hold for all $x,z \in \mathbb{R}^k$. Choose $x,z$ with $d(x,z)>C$, i.e.,
  $d_C(x,z)=C$. Then
  we receive $2 \alpha C \geq C$, i.e., $\alpha \geq \frac{1}{2}$.
\item{$\Longleftarrow$}  
  Let $\alpha \geq \frac{1}{2}$. Then we have that $d_C(x,z) \leq C
  \leq 2 \alpha C$ and \eqref{eq-triangle} is satisfied.
  \end{description}
\end{proof}

Note that the proof also shows that for a strictly increasing metric $d$ (such as $\ell_1$ or $\ell_2$) without cutoff, 
$(\mathbb{R}^k \cup \{\emptyset\},d_{\infty,\alpha})$ never is a metric space since \eqref{triangle} is always
violated for $y = \emptyset$ and $d(x,z)>2 \alpha C$. 
This is the reason why we do not treat location problems with empty set, but without cutoff. 

\subsection*{Relations between \eqref{classic}, \eqref{cutoff}, and \eqref{empty}}

We summarize a few observations on the relations between the optimal
values of the three problems.
\medskip

\begin{lemma}\label{lem:UpperBounds}
  We always have
  \begin{itemize}
  \item [(i)] $\mathcal{Z}_{C,\alpha}^\ast \leq \mathcal{Z}_C^\ast \leq \mathcal{Z}^\ast$
    \item [(ii)] $\mathcal{Z}_C^\ast \leq (n -1) \cdot C$
    \item [(iii)] $\mathcal{Z}_{C,\alpha}^\ast \leq
       C \cdot \min \{ n -1, n \cdot \alpha\}$.
  \end{itemize}
\end{lemma}

\begin{proof} \
  \begin{itemize}
  \item [(i)]
    Since $d^q_C(x,y) \leq d^q(x,y)$ we get $f_C(x) \leq f(x)$ for all $x \in \mathbb{R}^k$, hence also
      $\min_{x \in \mathbb{R}^k} f_C(x) \leq \min_{x \in \mathbb{R}^k} f(x)$ and $\mathcal{Z}_C^\ast \leq \mathcal{Z}^\ast$ holds.

      Furthermore, the empty barycenter increases the set of feasible solutions, i.e., \eqref{empty} is
      a relaxation of \eqref{cutoff}. We conclude  $\mathcal{Z}_{C,\alpha}^\ast \leq \mathcal{Z}_C^\ast$.
  \item [(ii)]
      Let $a \in \mathcal{A}$. This is a feasible barycenter with objective value of
      $f_C(a)=\sum_{a' \in \mathcal{A}} d^q_C(a,a') \leq 0 + (n-1)\cdot C$, hence an
      upper bound on \eqref{cutoff}.
    \item [(iii)]
      The empty barycenter is feasible and has an objective value of $f_{C,\alpha}=n \cdot \alpha \cdot C$, hence an upper bound on \eqref{empty}. Together with (i) and (ii), the result follows.
  \end{itemize}
\end{proof}

\section{Exploiting the local structure of \Nam$_C$} \label{sec:LocalStructure}

We start with some general properties of the cutoff. To this end we need some further notation.
\begin{de}
  Let $x \in \R^k$. Then 
  \begin{itemize}
  \item $\act_{C}(x):=\{a \in \cA: d^{q}(x,a) \leq C\}$ denotes the \emph{active points} w.r.t $x$ and $C$, and
  \item $\con_{C}(x):=\cA \setminus \act_{C}(x)$ denotes the \emph{constant points} w.r.t $x$ and $C$, i.e., the points whose distances remain locally constant.
  \end{itemize}  
  When we know the value of $C$ we just write $\act(x)$ and $\con(x)$.
\end{de}

We can now split the objective function into an active and a constant part,
\begin{eqnarray}
  f_C(x,\cA) & = & f_C(x,\act_{C}(x)) + f_C(x,\con_{C}(x)) \nonumber \\
            & = & \sum_{a \in \act_{C}(x)} d^{q}(x,a) + C \cdot |\con_{C}(x)| \nonumber\\
            & = & f(x,\act_{C}(x)) + C \cdot |\con_{C}(x)| \label{eq-split}.               
\end{eqnarray}           

This decomposition gives us a first basic result showing that the \Name\
problem with cutoff is equivalent
to a problem of type (\Nam), but w.r.t a subset of the existing points.

The following Lemma is an extension of Lemma $2$ of \cite{drezner1991facility}, who proved this result for $d\in \lbrace \ell_1, \ell_2 \rbrace$, but it is visible that the proof works more generally. For the sake of completeness we present a proof for any metric $d$ and any $q \geq 1$.

\begin{lemma}
  \label{lemma1}
  Let $\xi^\ast \in \cX^*_C$ be an optimal solution to (\Nam$_C(\cA)$). Then the following hold:
  \begin{enumerate}
    \item[(i)]  $\xi^\ast$ is an optimal solution to {\rm (\Nam($\act(\xi^\ast)$))}.
    \item[(ii)] All optimal solutions for {\rm (\Nam($\act(\xi^\ast)$))}
      are optimal solutions to     {\rm (\Nam$_C(\cA)$)}
      i.e., $\cX^*(\act(\xi^\ast)) \subseteq \cX^*_C(\cA)$.
  \end{enumerate}    
\end{lemma}  

\begin{proof} \
  \begin{enumerate}
    \item[ad (i)]
Let $\xi^\ast \in \R^k$ be a minimizer of $f_C(x,\cA)$, but assume $f(y,\act(\xi^\ast)) < f(\xi^\ast,\act(\xi^\ast))$
for some $y \in \R^k$. Due to \eqref{eq-split} we then receive $f_C(y,\cA) < f_C(\xi^\ast, \cA)$, a
  contradiction to the optimality of $\xi^\ast$.
\item[ad (ii)]
  For the second statement, take $\eta^\ast \in \cX^*(\act(\xi^\ast))$. We consider $\act(\xi^\ast)$ and
  $\con(\xi^\ast)=\cA \setminus \act(\xi^\ast)$ separately:
\begin{eqnarray*}
  f_C(\xi^\ast,\act(\xi^\ast)) &=& f(\xi^\ast,\act(\xi^\ast)) = f(\eta^\ast, \act(\xi^\ast)) \ \geq \ f_C(\eta^\ast,\act(\xi^\ast))\\
  f_C(\xi^\ast,\con(\xi^\ast)) &=& \sum_{a \in \con(\xi^\ast)} C \ \geq \ \sum_{a \in \con(\xi^\ast)} \min\{d^{q}(\eta^\ast,a),C\} \ = \ f_C(\eta^\ast,\con(\xi^\ast))
\end{eqnarray*}
and together we receive that $f_C(\eta^\ast,\cA) \leq f_C(\xi^\ast,\cA)$, hence $\eta^\ast$ is also optimal.
\end{enumerate}
\end{proof}

This result is one of the main ideas needed for Algorithm~\ref{algo:Drezner} and its improved versions
which are described next. We state the approach of \cite{drezner1991facility} for our special case
of cut off distances. Note that the versions of \cite{aloise2012improved} and \cite{venkateshan2020note}
are not relevant for this setting.
\medskip

\SetKwInOut{input}{Input}\SetKwInOut{output}{Output}
\begin{algorithm}[ht]
  \caption{Algorithm for \eqref{cutoff}, based on \cite{drezner1991facility}.}
  \label{algo:Drezner}  
  \input{Set $\mathcal{A} = \lbrace a_1, \ldots, a_n \rbrace$, cutoff $C>0$}
  \output{A barycenter $\xi^{\ast}$ of $\eqref{cutoff}$, objective function value $\mathcal{Z}_C^\ast$}
  \BlankLine
  Set $\xi^{\ast} \leftarrow a_1$, $\mathcal{Z}_C^\ast \leftarrow \infty$\;
  \For{$i \leftarrow 1$ \KwTo $(n-1)$}{
    Inner Loop\;
  }
  \Return{$\xi^{\ast}, \mathcal{Z}_C^\ast$}\ 
  
  \BlankLine 
  \nonl \hrulefill \ 

  \nonl Inner Loop:\ 
  
  \For{$j \leftarrow (i+1)$ \KwTo $n$}{
    \If{$d^q(a_i,a_j) \leq 2^q\cdot C$}{
      Compute the centers $c_1, c_2$ of the balls with radius $C$ that fulfill $d^q(c_1,a_i) = d^q(c_1,a_j) = d^q(c_2,a_i) = d^q(c_2,a_j) = C$\;
      \For{$k \leftarrow 1$ \KwTo $2$}{
        Set $S := \lbrace a \in \mathcal{A}\ \vert \ d^q(c_k,a) \leq C \rbrace$\;
        Compute for the four sets $S, S\setminus \lbrace a_i \rbrace, S\setminus \lbrace a_j \rbrace, S\setminus \lbrace a_i, a_j \rbrace$ the barycenters $\xi_1, \ldots \xi_4$ and the corresponding objective function values $d_1, \ldots d_4$\;
        $l \leftarrow \underset{l \in \lbrace 1,\ldots, 4 \rbrace}{\min}  d_l$\;
        \If{$d_l < \mathcal{Z}_C^\ast$}{
          $\mathcal{Z}_C^\ast \leftarrow d_l$, $\xi^{\ast} \leftarrow \xi_l$\;
        }
      }
    }
  }
\end{algorithm}

The following observations are true in the plane. 
We know from Lemma~\ref{lemma1} that any optimal solution to \eqref{cutoff} is a solution to
(\Nam($A$)) for a subset $A \subseteq \mathcal{A}$. Algorithm~\ref{algo:Drezner} uses brute force
to calculate the optimal solutions for these subsets. But instead of enumerating all
theoretically possible $2^{n}-1$ subsets, \cite{drezner1991facility}
use the following geometric observation to cut down the
number of subsets to look at: Say we have an optimal solution $\xi^\ast \in \mathcal{X}^\ast_C$.
Set $A := \act(\xi^\ast)$. $A$ is contained in a ($2$-dimensional) ball around $\xi^\ast$ with radius $C$. This
ball can be "moved" so that two points, $a_1,a_2$ of $\mathcal{A}$ lie on the circumference of the
ball and all points of $A$ are still inside. We hence can restrict our search to all balls with radius $C$
that are defined by two points of $\mathcal{A}$ on its circumference.
\cite{drezner1991facility} proved that there are at most $\mathcal{O}(n^2)$ of these balls, so we only need to solve (\Nam($A$)) for $\mathcal{O}(n^2)$ subsets.
The arguments of the proof hold for all norm metrics $d$ and all $q \geq 1$, although \cite{drezner1991facility} did not state these cases explicitly.

\begin{theo}[\cite{drezner1991facility}]
  Let $\mathcal{A} \subseteq \mathbb{R}^2$, let $d$ be a norm-metric and say we can solve \eqref{classic} in $h(n)$ time. 
  Then Algorithm~\ref{algo:Drezner} solves the problem \eqref{cutoff} in $\mathcal{O}(n^2\cdot h(n))$ time.
\end{theo}

\begin{proof}
  The result was proven in \cite{drezner1991facility} for $d \in \{\ell_1,\ell_2\}$. The proof is based on two arguments. First the solution of \eqref{cutoff} is a solution to (Bar(A)) for some $A \subseteq \mathcal{A}$. And second the number of these subsets $A$ we need to check for the optimal solution is of order $n^2$. 
  Lemma~\ref{lemma1} states the first argument for any metric $d$ and any $q \geq 1$. 
  And the second argument follows directly from the proof of Theorem 1 in \cite{drezner1991facility}. The argument in the proof works for any ball defined by a norm-metric $d$. A ball that is defined by $d^q$ only differs in its radius from a ball that is defined by $d$. So the number of candidate subsets $A$ is of order $n^2$ for any norm-metric $d$ and any $q \geq 1$. 
\end{proof}  

\begin{remark}
  Algorithm~\ref{algo:Drezner} is presented only for finite subsets $\mathcal{A}$ of $\mathbb{R}^2$. 
  The method of cutting down the number of $2^n -1$ theoretically possible subsets of $\mathcal{A}$ to a polynomial number of subsets also works in $\mathbb{R}^k$ for $k\geq 3$.
  The $k$-dimensional ball with radius $C$ is uniquely defined by $k$ points that define a $k-1$ dimensional hyperplane.
  For $k=2$ we need two points that are not identical, for $k=3$ we need three points that are not collinear.
  With the same arguments as for the the $2$-dimensional case, the number of candidate sets is bound by $\mathcal{O}(\binom{n}{k})$ = $\mathcal{O}(n^k)$.
  Therefore Algorithm~\ref{algo:Drezner} can be solved in $k$ dimensions in $\mathcal{O}(n^k\cdot h(n))$ time.
\end{remark}

We additionally suggest the following improvement that 
is obtained by replacing lines $1$-$5$ by Algorithm~\ref{algo:improvement}: Instead of investigating all $a_i,a_j$ with
$d(a_i,a_j) \leq 2^q C$ we sort out points $a_i$ for which we can be sure that 
they will not lead to a solution that improves our current best objective function value. The sorting out is based on the following lemmas.

\begin{lemma}\label{lem:manyfarpts}
  Let $a\in \mathcal{A}$ and $A^\prime:=\lbrace a^\prime \in \mathcal{A} \ \vert \ d^q(a,a^\prime) > 2^qC \rbrace$. Let $x\in \mathbb{R}^k$ s.t. $d^q(x,a) \leq C$. Then 
  \begin{itemize}
    \item[(i)] $A^\prime \subseteq \con(x)$,
    \item[(ii)] $f_C(x) \geq C\cdot \vert A^\prime \vert$.
  \end{itemize}
  
\end{lemma}

\begin{proof} \ 
  \begin{itemize}
    \item[ad (i)] 
    Take $a^\prime \in A^\prime$.
    We know by definition of $A^\prime$ that $$d^q(a,a^\prime) > 2^q C \Rightarrow d(a,a^\prime) > 2 \sqrt[q]{C}.$$
    We also know that $$d^q(a,x) \leq C \Rightarrow d(a,x) \leq \sqrt[q]{C}.$$
    With the triangle inequality we get that $$\underbrace{d(a,a^\prime)}_{> 2 \sqrt[q]{C}} - \underbrace{d(a,x)}_{\leq \sqrt[q]{C}} \leq d(x,a^\prime).$$
    The left side of the inequality is $> \sqrt[q]{C}$. Hence $\sqrt[q]{C} < d(x,a^\prime) \Rightarrow C < d^q(x,a^\prime) \Rightarrow a^\prime \in \con(x)$.
    \item[ad (ii)] Now we know from (i) that $A^\prime \subseteq \con(x)$. Therefore $f_C(x) \geq C\cdot \vert \con(x) \vert \geq C\cdot \vert A^\prime \vert$.
  \end{itemize}
  
\end{proof}

\begin{lemma}\label{lem:skipa}
  Let $a \in \mathcal{A}$, let $z=f_C(x)$ for some $x\in \mathbb{R}^k$.
  Let $A^\prime:=\lbrace a^\prime \in \mathcal{A} \ \vert \ d^q(a,a^\prime) > 2^qC \rbrace$ like in Lemma~\ref{lem:manyfarpts}. If $C\cdot \vert A^\prime \vert > z$ then 
  no set $S \ni a$ constructed in Algorithm~\ref{algo:Drezner} will lead to an optimal solution $\xi^\ast \in \mathcal{X}^\ast_C$.
\end{lemma}
\begin{proof}
  Let $\xi^\ast \in \mathcal{X}^\ast_C$ be an optimal solution to \eqref{cutoff}.
  We know from Lemma~\ref{lemma1} that any optimal $\xi^\ast \in \mathcal{X}^\ast_C$ is a solution of (Bar($ \act(\xi^\ast)$)). 
  Say we have constructed a set $S$ containing $a\in \mathcal{A}$ in line $10$ of Algorithm~\ref{algo:Drezner}. Suppose there is a $\xi \in \mathcal{X}^\ast_C$, such that $\act(\xi) = S$. 
  We know then that $f_C(\xi) \geq C\cdot |\mathcal{A}\setminus S|$. With Lemma~\ref{lem:manyfarpts}(i) and the construction of the set $S$ we know that $\mathcal{A}\setminus S \supseteq A^\prime$. Therefore $f_C(\xi) \geq C\cdot |\mathcal{A}\setminus S| \geq C\cdot \vert A^\prime \vert > z \geq \mathcal{Z}^\ast_{C}$. Therefore no set $S$ that we construct 
  in Algorithm~\ref{algo:Drezner} that contains the point $a$, nor any of its subsets are the active set of an optimal solution.
\end{proof}

\begin{algorithm}[!ht]
  \caption{Improvement of Algorithm~\ref{algo:Drezner}}
  \label{algo:improvement}
  \input{Set $\mathcal{A} = \lbrace a_1, \ldots, a_n \rbrace$, cutoff $C>0$}
  \output{A barycenter $\xi^{\ast}$ of $\eqref{cutoff}$, objective function value $\mathcal{Z}_C^\ast$}
  \BlankLine
  Set $\xi^{\ast} \leftarrow a_1$, $\mathcal{Z}_C^\ast \leftarrow \infty$, $\mathcal{B} \leftarrow \mathcal{A}$\;
  \For{$i \leftarrow 1$ \KwTo $(n-1)$}{
    $continue$ $\leftarrow$ $true$\;
    $m \leftarrow \vert \lbrace a \in \mathcal{B} \ \vert \ d^q(a_i,a) \leq 2^qC \rbrace \vert$\;
    \If{$(n-m) \cdot C \geq \mathcal{Z}_C^\ast$}{
      $\mathcal{B} \leftarrow \mathcal{B}\setminus \lbrace a_i \rbrace$\;
      $continue$ $\leftarrow$ $false$\;
      }
    \If{continue}{
      Inner Loop\;
    }
  }
  \Return{$\xi^{\ast}, \mathcal{Z}_C^\ast$}
\end{algorithm}

\begin{theo}\label{lem:reduction1}
  Let $\mathcal{A} \subseteq\mathbb{R}^2$, let $d$ be a norm-metric and say we can solve \eqref{classic} in $h(n)$ time. Then Algorithm~\ref{algo:improvement} solves the problem \eqref{cutoff} in $\mathcal{O}(n^2\cdot h(n))$ time.
\end{theo}

\begin{proof}
  We have to prove two things: first the runtime and second the correctness.
  \\
  First: The calculation of $m$ takes $\mathcal{O}(n)$ time and hence does not increase the runtime of the algorithm.
  \\
  Second: 
  For the correctness we have to prove that although we skip the \emph{Inner Loop} for some $a_i$ we still compute the optimal solution.
  Lemma~\ref{lem:skipa} implies that we can skip any point $a_i$ if for $A^\prime:=\lbrace a^\prime \in \mathcal{A} \ \vert \ d^q(a_i,a^\prime) > 2^qC \rbrace$ the value $|A^\prime| \cdot C = (n-m)\cdot C$ is larger than the current best objective function value. 
  In addition the proof of Lemma~\ref{lem:skipa} yields that $a_i \not\in \act(\xi^\ast)$ for $\xi^\ast \in \mathcal{X}^\ast_C$. It is therefore justified to permanently remove $a_i$ from the candidate set of potentially active points in line $4$ of Algorithm~\ref{algo:improvement}.
\end{proof}

Later in Section~\ref{sec:emptyCenter}, where we solve \eqref{empty}, we can further reduce the computation time 
with the knowledge that also any point $a_i$ 
for which $m \leq (1-\alpha)\cdot n$ can also be disregarded,
compare Lemma~\ref{lem:largeActiveXi}.

\subsubsection*{Other consequences of Lemma~\ref{lemma1}}

Apart from its algorithmic implication, Lemma~\ref{lemma1} has several other consequences
since it transfers properties that depend on the local structure from \eqref{classic} to \eqref{cutoff}. 
This holds for properties which are only based on the metric and on the existing facilities. 
Such properties then also hold for subsets of the existing facilities,
and in particular for $\act(\xi^*)$ where $\xi^*$ is an optimal solution of \eqref{cutoff}. 
A first example of such a condition which will be used later in Theorem~\ref{thm:ignorecut} 
is the property \eqref{convAssumption} for \eqref{classic} that there always exists an optimal solution
$\xi^*$ to the barycenter problem which is contained in the convex hull $\conv(\cA)$ of the existing
facilties. 

\begin{equation*}\label{convAssumption}
  \text{ For any subset} A \subseteq \mathcal{A} : \cX^*(A) \cap \conv(A) \neq \emptyset \tag{conv}
\end{equation*}

If \eqref{convAssumption} holds for \eqref{classic} then it also holds for \eqref{cutoff}.

\begin{lemma}\label{cor:convexhulls}
  If \eqref{convAssumption} then $\mathcal{X}_C^{\ast}(\mathcal{A}) \cap \conv(\mathcal{A}) \neq \emptyset$.
\end{lemma}

\begin{proof}
  \ 
  Let $\xi^\ast \in \cX^*_C(\cA)$. From Lemma~\ref{lemma1} we know that
  $\xi^\ast$ is optimal for \Nam$(\act(\xi^\ast))$.
  There exists $\eta^\ast \in \cX^*(\act(\xi^\ast))$
  with $\eta^\ast \in \conv(\act(\xi^\ast)) \subseteq \conv(\cA)$.
  From the second part of Lemma~\ref{lemma1} we know that
  $\eta^\ast \in \cX^*_C(\cA)$. Together, the result follows.
\end{proof}

Condition \eqref{convAssumption} is satisfied for many location problems. We list
cases in which it holds below.

\begin{itemize}
\item In $\mathbb{R}^2$ (and in $\mathbb{R}^1$) \eqref{convAssumption} holds for all
  distances $d$ dervied from norms \cite{plastria1984localization}.
\item For $k>2$, \eqref{convAssumption} only holds in general if $d$ is a norm which is linearly equivalent
  to the $\ell_2$-norm \cite{plastria1984localization}.
\item \eqref{convAssumption} holds for $d=\ell_2^2$, since the (unique) optimal solution of \eqref{classic} is in
  this case the coordinate-wise mean of the points in $\cA$ which is always contained in $\conv(\cA)$. 
\end{itemize}

There are many other examples of conditions which can be transferred form \eqref{classic}
to \eqref{cutoff}. Among them are:
\begin{itemize}
\item There exists a finite candidate set for \eqref{cutoff} if $d$ is derived from a
  polyhedral norm. This candidate set can be found by using the
  intersection points of the fundamental directions.
\item For problems \eqref{cutoff} with restricted set $R$ all optimal solutions are either optimal
  solutions for the unrestricted problem or are contained
  in the boundary of $R$.
\end{itemize}
\bigskip

Above we stated the property \eqref{convAssumption}, which will enable us to make a connection between \eqref{classic} and \eqref{cutoff} in Theorem~\ref{thm:ignorecut}.
We now state a weaker assumption that also allows for a connection between \eqref{classic} and \eqref{cutoff}.

\begin{align*}\label{ballAssumption}
  &\text{There exists a ball } B = B(x,r) \text{ such that:}  
  \\
  &\text{for all } A \subseteq \mathcal{A}: \mathcal{X}^{\ast}(A) \cap B \neq \emptyset
  \tag{B}
\end{align*}

\begin{lemma}
  Condition \eqref{ballAssumption} implies that $\mathcal{A} \subseteq B$.
\end{lemma}
\begin{proof}
  For any point $a \in \mathcal{A}$ the singleton $\lbrace a \rbrace$ is a subset of $\mathcal{A}$. The optimal solution $\xi^\ast$ of (Bar$(\lbrace a \rbrace)$) is $\xi^\ast = a$. Therefore $B$ must contain all points $a \in \mathcal{A}$.
\end{proof}

The set $\mathcal{A}$ is finite. That means there are only finitely many subsets $A \subseteq \mathcal{A}$ and a ball $B$ that fulfills \eqref{ballAssumption} always exists. 
We consider the smallest one.

\begin{de}
  We define a \emph{ball} with center $x \in \mathbb{R}^k$ and radius $r>0$ by $B := B(x,r) := \lbrace y\in \mathbb{R}^k \; \vert \; d(x,y) \leq r \rbrace$.
  We denote by 
  \\
  $B_0 = B(x_0,r_0)$ 
  a \emph{smallest ball} (in terms of radius) 
  that fulfills \eqref{ballAssumption}.
\end{de}

We can now make a connection between the optimal objective function values $\mathcal{Z}^\ast$ and $\mathcal{Z}^\ast_C$.

\begin{theo}\label{thm:2r0}
  If $2r_0 \leq \sqrt[q]{C}$, then $\mathcal{Z}^\ast = \mathcal{Z}_C^\ast$
\end{theo}

\begin{proof}
  Take an optimal solution $\xi^\ast$ of \eqref{cutoff} inside $B_0$ and an optimal solution $\eta^\ast$ of \eqref{classic} inside $B_0$. 
  Both solutions must exist due to \eqref{ballAssumption}. We prove that $\mathcal{Z}^\ast = f(\eta^\ast) = f_C(\xi^\ast) = \mathcal{Z}^\ast_C$.
  
  We know that $\xi^\ast \in B_0$. That means for every $a \in \mathcal{A}$ that $d(\xi^\ast,a) \leq d(\xi^\ast,x_0) + d(x_0,a) \leq 2r_0$. 
  Then $d^q(\xi^\ast,a) \leq (2r_0)^q \leq C$ and hence $f_C(\xi^\ast) = f(\xi^\ast)$. 
  Analogously we get that $f(\eta^\ast) = f_C(\eta^\ast)$.
  From the optimality of both $\xi^\ast$ and $\eta^\ast$ it follows that $f_C(\xi^\ast) \leq f_C(\eta^\ast)$ and 
  $f(\eta^\ast) \leq f(\xi^\ast)$ and therefore $f(\xi^\ast) = f_C(\xi^\ast) \leq f_C(\eta^\ast) = f(\eta^\ast) \leq f(\xi^\ast)$, i.e. $f(\xi^\ast) = f(\eta^\ast)$.
\end{proof}

We will see in the next section in Lemma~\ref{lem:atob} that $\mathcal{Z}^\ast = \mathcal{Z}_C^\ast$ also implies that $\mathcal{X}^{\ast} \subseteq \mathcal{X}^{\ast}_C$.

\section{Comparing \Nam$_C$ and \Nam}\label{sec:cutoff}

In this section we have a closer look at the barycenter problem with cutoff in comparison to the barycenter problem without cutoff. 
In general, problem \eqref{classic} has an easier structure than problem~\eqref{cutoff}.
While \eqref{classic} is a convex problem for every norm-metric $d$, the cutoff destroys convexity and can, e.g., lead to non-connected optimal solution sets. 
In the following we identify conditions under which solving \eqref{classic} gives us the objective function value of \eqref{cutoff} or
even an optimal solution of the latter. 
\begin{itemize}
\item[(a)] $(\Nam)$ has the same objective function value as $(\Nam_C)$, i.e. $\cZ^*=\mathcal{Z}^\ast_C$. 
\item[(b)] Any solution to $(\Nam)$ is a solution to $(\Nam_C)$, i.e., $\mathcal{X}^{\ast} \subseteq \mathcal{X}^{\ast}_C$.
\end{itemize}

If the second condition holds then it is sufficient to solve $(\Nam)$. We first show that
condition (b) already follows from (a) (but not vice versa), so either condition is useful. 

\begin{lemma}\label{lem:atob}
If condition $(a)$ holds then $(b)$ holds as well.
\end{lemma}

\begin{proof}
 Let $\xi^{\ast} \in \mathcal{X}^{\ast}$ be an optimal solution to \eqref{classic} and
  $\eta^{\ast} \in \mathcal{X}^{\ast}_C$ an optimal solution to \eqref{cutoff}. 
  From condition (a) we know that $f(\xi^{\ast}) = f_C(\eta^{\ast})$. In order
  to show that $\xi^{\ast} \in \mathcal{X}^{\ast}_C$, we compute
  \[f_C(\xi^{\ast}) \leq f(\xi^{\ast})=f_C(\eta^{\ast})\leq f_C(\xi^{\ast}). \]
Consequently, $f_C(\xi^{\ast}) = f_C(\eta^{\ast})$ and  hence $\xi^{\ast} \in \mathcal{X}^{\ast}_C$.
\end{proof}

With this result we know, that as soon as condition $(a)$ holds, we can solve the problem $(\Nam)$ and automatically get a solution to $(\Nam_C)$.

The implication $(b) \Rightarrow (a)$ is not true in general, as a simple one-dimensional example shows: 
\begin{example}[Counterexample to (b) $\Rightarrow$ (a)]\label{ex:XinXc1dim}
  Let 4 points in $\R$ be given, $a_1=a_2=0$, $a_3=C + \varepsilon$ and $a_4=-C-\varepsilon$ and let $q=1$. 
  The solution to both problems, \eqref{classic} and \eqref{cutoff} is $\mathcal{X}^{\ast} = \mathcal{X}^{\ast}_C = \lbrace 0 \rbrace$, but $\mathcal{Z}^\ast =2(C+\varepsilon) > 2C = \mathcal{Z}^\ast_C$.
\end{example}

We now show that for a set $\mathcal{A}$ with a large diameter, condition $(a)$ is not met. To this end,
we use that for two points, a barycenter is given by their arithmetic mean.

\begin{lemma}\label{lem:halfisoptimal}
  For two points $a,b \in \mathbb{R}^k$, for any $\ell_p$-norm and for all $q\geq 1$, a minimizer of $\Vert a-x \Vert^q + \Vert b-x \Vert^q$ is $x = \frac{a+b}{2}$.
\end{lemma}

\begin{proof}
  For $\ell_p$-norms this can be treated as a one-dimensional problem, since the optimal solutions are on the line between $a$ and $b$. W.l.o.g say $a=0, b=1$. Every other case follows by scaling.
  The resulting objective function is $f(x)=x^q + (1-x)^q$ whose minimium is attained at
    $\bar{x}=\frac{1}{2}$.
\end{proof}

The next theorem identifies cases in which condition (a) does not hold; i.e., cases in which the objective
function value of \eqref{cutoff} is strictly smaller than that of \eqref{classic}.

\begin{theo}\label{thm:largediam}
  \eqref{cutoff} has a strictly smaller objective function value than \eqref{classic} in the following
    two cases:
  \begin{itemize}
  \item [(i)] $\diam(\mathcal{A}) > 2C$, $q=1$ and $d$ is a metric,
  \item [(ii)] $\diam(\mathcal{A}) > 2\sqrt[q]{C}$, $q>1$ and $d$ is derived from an $\ell_p$-norm.
\end{itemize}
\end{theo}

\begin{proof}
Since $\mathcal{A}$ is finite there exist two points $a,b \in \mathcal{A}$ such that $d(a,b) = \diam(\mathcal{A}) > 2\sqrt[q]{C}$. 
\begin{itemize}
\item[ad (i):] for any point $x \in \mathbb{R}^k$ the triangle inequality directly gives
  $d(x,a) + d(x,b) \geq d(a,b) > 2C$.
\item[ad (ii):] we use Lemma~\ref{lem:halfisoptimal}, namely 
  that a minimizer of $\Vert a-x \Vert^q + \Vert b-x \Vert^q$ is given by $\bar{x} = \frac{a+b}{2}$. 
  We receive that for any point $x \in \R^k$: 
  \begin{eqnarray*}
    d^q(a,x) + d^q(x,b) & = & \Vert a-x \Vert^q + \Vert x-b \Vert^q
    \geq \Vert a- \bar{x} \Vert^q + \Vert \bar{x} - b \Vert^q
    \\ 
     & = & 2^{-(q-1)} \Vert a-b \Vert^q > 2^{-(q-1)}\cdot 2^qC = 2C.
\end{eqnarray*}                              
\end{itemize}
In both cases, at least one of the distances $d^q(a,x)$ or $d^q(x,b)$ is larger than the cutoff $C$
for any $x \in \R^k$. This holds especially for a barycenter $\xi^\ast \in \mathcal{X}^{\ast}$. Therefore
\begin{align*}
  f(\xi^\ast) = \sum_{a\in \cA} d^q(\xi^\ast,a) > \sum_{a\in \cA} \min \lbrace d^q(\xi^\ast,a), C\rbrace = f_C(\xi^\ast).
\end{align*}
Let $\eta^\ast \in \mathcal{X}^{\ast}_C$. We know that $f_C(\xi^\ast) \geq f_C(\eta^\ast)$.
Hence,
\[ \cZ^*=f(\xi^\ast) > f_C(\xi^\ast) \geq f_C(\eta^\ast)=\cZ^*_C.\]
\end{proof}

The next theorem identifies a setting in which condition (a) and hence also condition (b) hold, i.e., in which
\eqref{classic} can be used to obtain an optimal solution to \eqref{cutoff}.

\begin{theo}\label{thm:ignorecut}
  Consider a location problem \eqref{classic} which satisfies property \eqref{convAssumption}. 
  If $\diam(\mathcal{A}) \leq \sqrt[q]{C}$, then
  $\mathcal{Z}^\ast = \mathcal{Z}^\ast_C$ and $\mathcal{X}^{\ast} \subseteq \mathcal{X}^{\ast}_C$.
\end{theo}
  
\begin{proof}
  Let $\xi^\ast \in \mathcal{X}^{\ast}$ be an optimal solution to \eqref{classic} and
  $\eta^{\ast} \in \mathcal{X}^{\ast}_C$ be an optimal solution to \eqref{cutoff}.
  By \eqref{convAssumption} and Corollary~\ref{cor:convexhulls} we may choose both,
  $\xi^\ast$ and $\eta^\ast \in \conv(\mathcal{A})$.

  For any point $x \in \conv(\mathcal{A})$ and for all $a \in \mathcal{A}$ we have that
  $d(x,a) \leq \diam(\mathcal{A}) \leq \sqrt[q]{C}$, therefore $d^q(x,a) \leq C$ and thus $f(x) = f_C(x)$. In particular, we receive
  \begin{eqnarray*}
    f(\xi^\ast) & = & f_C(\xi^\ast)\\
    f_C(\eta^{\ast}) & = & f(\eta^{\ast}).
  \end{eqnarray*}                         
  Hence we obtain  
  \[ f_C(\xi^\ast) = f(\xi^\ast) \leq f(\eta^{\ast}) = f_C(\eta^{\ast}) \leq f_C(\xi^\ast),\]
  i.e., $\cZ^*=f(\xi^\ast) = f_C(\eta^{\ast})=\cZ^*_C$. By Lemma~\ref{lem:atob}, we also get
  $\mathcal{X}^{\ast} \subseteq \mathcal{X}^{\ast}_C$.
\end{proof}

We hence know that $\cZ^*=\cZ^*_C$ if the diameter of the set $\cA$ is smaller or equal to $\sqrt[q]{C}$ and that
$\cZ^*>\cZ^*_C$ if the diameter is greater than $2\sqrt[q]{C}$. The following examples demonstrate that
for the remaining cases, $\sqrt[q]{C} < \diam(\cA) \leq 2\sqrt[q]{C}$ everything may happen.

\begin{example}\label{ex:XnotinXc1dim}[Example where (a) holds for $q=1$ and $C < \diam(\cA) \leq 2C$]
  Let $n+2$ points in $\mathbb{R}$ be given, $n \geq 2$. Say $a_1=\ldots=a_n=0, a_{n+1}= C-\varepsilon, a_{n+2}= -(C-\varepsilon)$. The diameter of this set is $\diam(\mathcal{A}) = d(a_{n+1},a_{n+2}) = 2C-2\varepsilon$.
  Now $\mathcal{X}^{\ast} = \mathcal{X}^{\ast}_C = \lbrace 0 \rbrace$ and $\mathcal{Z}^\ast = \mathcal{Z}^\ast_C$.
\end{example} 

In this case the solution of $(\Nam)$ is also a solution of $(\Nam_C)$. But there are simple examples, where $\mathcal{X}^{\ast} \not\subseteq \mathcal{X}^{\ast}_C$.

\begin{example}\label{ex:XnotinXc1norm}[Example where (b) does not hold for $q=1$, $d=\ell_1$ and $C < \diam(\cA) \leq 2C$]
  \begin{center}
    \begin{tikzpicture}[scale = 1]
      \draw [draw=black] (-1.6,-1.6) rectangle ++(7.2,3.2);
      \tkzInit[xmax=5,ymax=1.5,xmin=-1.5,ymin=-1.5]
      \draw [->, color=gray, dotted] (-1.4,0) -- (5,0);
      \draw [->, color=gray, dotted] (0,-1.4) -- (0,1.5);
      \node  (1) at (-1, 1) {$a_1$};
      \node  (2) at (-1, -1) {$a_2$};
      \node  (3) at (0, 0) {$a_3$};
      \node  (4) at (4, 0.5) {$a_4$};
      \node  (5) at (4, -0.5) {$a_5$};
      \node  (6) at (-1, 0) {$\eta$};
    \end{tikzpicture}
  \end{center} 

Let $5$ points $a_1$ to $a_5$ in $\mathbb{R}^2$ be given with coordinates $a_1=(-\frac{\varepsilon}{2}/\frac{\varepsilon}{2}), a_2=(-\frac{\varepsilon}{2}/ -\frac{\varepsilon}{2}), a_3=(0/0), a_4=(C-\frac{\varepsilon}{4}/\frac{\varepsilon}{4}), a_5=(C-\frac{\varepsilon}{4}/-\frac{\varepsilon}{4})$ for some $\varepsilon > 0$.
The placement is pictured above. We get the following distances: $d(a_3,a_4)=d(a_3,a_5)=C$, $d(a_1,a_3)=d(a_2,a_3) = \varepsilon$. The diameter of this set is $d(a_1,a_5) = d(a_2,a_4) = d(a_1,a_3) + d(a_3,a_5) = C+\varepsilon$. The optimal solution of \eqref{classic} is $a_3$ 
with $\mathcal{Z}^\ast = 2C + 2\varepsilon = f_C(a_3)$. For $\eta = (-\frac{\varepsilon}{2}/0)$ we get $f_C(\eta) = \frac{3\varepsilon}{2}+2C < f_C(a_3)$. Thus $a_3 \in \mathcal{X}^\ast$ but $a_3 \not\in \mathcal{X}_C^\ast$.

\end{example}

\begin{example}\label{ex:triangle}[Example where (b) does not hold for $q=1$, $d=\ell_2$ and $C < \diam(\cA) \leq 2C$]
  \begin{center}
    \begin{tikzpicture}[scale = 1]
      \draw [draw=black] (-1.6,-1.6) rectangle ++(7.2,3.2);
      \tkzInit[xmax=5,ymax=1.5,xmin=-1.5,ymin=-1.5]
      \draw [->, color=gray, dotted] (-1.4,0) -- (5,0);
      \draw [->, color=gray, dotted] (0,-1.4) -- (0,1.5);
      \node  (1) at (-1, 1) {$a_1$};
      \node  (2) at (-1, -1) {$a_2$};
      \node  (3) at (0, 0) {$\xi$};
      \node  (4) at (4, 0) {$a_3$};
      \node  (5) at (-1, 0) {$\eta$};
    \end{tikzpicture}
  \end{center}

Let $3$ points $a_1$ to $a_3$ in $\mathbb{R}^2$ be given with coordinates $a_1=(-\frac{C}{2}/\frac{\sqrt{3}C}{2}),a_2=(-\frac{C}{2}/-\frac{\sqrt{3}C}{2}),a_3=(C/0)$. 
They form an equilateral triangle with sidelength $\sqrt{3}C$. The placement is sketched above. 
The diameter of this set is equal to the length of one side of the triangle which is larger than $C$ but smaller than $2C$. 
The optimal solution to \eqref{classic} is $\xi = (0/0)$ with $\mathcal{Z}^\ast=3C = f_C(\xi)$. But for $\eta = (-\frac{C}{2}/0)$ we get $f_C(\eta) = (\sqrt{3} + 1)C < f_C(\xi)$. Therefore $\xi \in \mathcal{X}^\ast$ but $\xi \not\in \mathcal{X}_C^\ast$. 
(Optimal solutions to \eqref{cutoff} would be each of the points $a_1$ to $a_3$ with $\mathcal{Z}^\ast_C = 2C$.)
\end{example}

\begin{example}\label{ex:triangleq2}[Example where (b) is not true for $q=2$, $d=\ell_2$ and $\sqrt{C} < \diam(\mathcal{A}) \leq 2\sqrt{C}$]
Take the same situation as in Example~\ref{ex:triangle}, but for simplicity set $C=1$. 
The diameter of this set is equal to the length of one side of the triangle which is $\sqrt[2]{3} \approx 1.732$ and therefore smaller than $2$. 
The optimal solution to \eqref{classic} is $\xi = (0/0)$ with $\mathcal{Z}^\ast =3 = f_C(\xi)$. 
But for $\eta = \left(-\frac{1}{2}/0 \right)$ we get $f_C(\eta) = \left(2\cdot \left(\frac{\sqrt{3}}{2}\right)^2 + 1\right) = \frac{5}{2} < f_C(\xi)$. 
Therefore $\xi \in \mathcal{X}^\ast$ but $\xi \not\in \mathcal{X}_C^\ast$.
\end{example}

We remark that in the last three examples above we have $f(\xi^\ast) = f_C(\xi^\ast)$ for the
(respective) optimal solution $\xi^\ast$ to \eqref{classic} but still $\xi^\ast \not\in \mathcal{X}^{\ast}_C$,
i.e., this solution is not optimal for \eqref{cutoff}. 
\medskip

In the following table we summarize the results for metrics $d$ and $q\geq 1$: 
\begin{table}[!ht]
  \begin{tabular}{|c|c|c|c|}
    \hline
    & $\diam(\mathcal{A}) \leq \sqrt[q]{C}$ & $\sqrt[q]{C} < \diam(\mathcal{A}) \leq 2\sqrt[q]{C}$ & $\diam(\mathcal{A}) > 2\sqrt[q]{C}$
    \\
    \hline
    $(a)\ \mathcal{Z}^\ast = \mathcal{Z}^\ast_C$ & holds if \eqref{convAssumption}, & may or may not hold,  & never for $q=1$,
    \\
     & see Thm~\ref{thm:ignorecut} &  see Examples~\ref{ex:XnotinXc1dim} to~\ref{ex:triangleq2} & never for $q>1$ for $\ell_p$-norms,
    \\
     &   &  & see Thm~\ref{thm:largediam}
    \\
    \hline
    $(b)\ \mathcal{X}^\ast \subseteq \mathcal{X}^\ast_C$ & holds if \eqref{convAssumption}, & may or may not hold, & may or may not hold,
    \\
    & follows from Lem~\ref{lem:atob} & see Examples~\ref{ex:XnotinXc1dim} to~\ref{ex:triangleq2} & see Examples~\ref{ex:XinXc1dim} and~\ref{ex:XnotinXc1norm}
    \\
    \hline
  \end{tabular}
  \label{table:summary1}
\end{table}

Furthermore, we have seen in Theorem~\ref{thm:2r0} that (a) and (b) are always true if $2r_0 \leq \sqrt[q]{C}$, where $r_0$ is the radius of a smallest ball $B$ such that $\mathcal{X}^\ast(A) \cap B \neq \emptyset$ for all $A \subseteq \mathcal{A}$.
\medskip

For a very small cutoff $C$ relative to the distances between the points of $\mathcal{A}$ 
and for a large cutoff compared to the diameter of $\mathcal{A}$ we can say something about the optimal solutions to \eqref{cutoff}:

\begin{lemma}
  Let $\xi^\ast$ be an optimal solution to \eqref{cutoff}.
  \begin{itemize}
    \item [(i)] If $C < \frac{1}{2^q} \underset{a_1\neq a_2 \in \mathcal{A}}{\min} d^q(a_1,a_2)$ we have $\xi^\ast \in \mathcal{A}$ and $|\act(\xi^\ast)| = 1$. 
    \item [(ii)] If $\sqrt[q]{C} \geq 2\cdot \diam(\mathcal{A})$ we have $|\act(\xi^\ast)| = n$, implying $\mathcal{Z}^\ast = \mathcal{Z}^\ast_C$ and $\mathcal{X}^\ast = \mathcal{X}^\ast_C$. 
  \end{itemize}
\end{lemma}

\begin{proof}
  \begin{itemize}
    \item [(i)] We show that there is no better barycenter than a point $a \in \mathcal{A}$. The cutoff is smaller than the shortest distance between two points of $\mathcal{A}$. Therefore for any $a\in \mathcal{A}:$ $f_C(a,\mathcal{A}) = (n-1)\cdot C$. 
    \\
    Suppose $|\act_C(a)|=1$ and there is a point $\xi \in \mathbb{R}^k: f_C(\xi,\mathcal{A}) < (n-1)\cdot C$, then $\vert \act(\xi) \vert \geq 2$. 
    Take two different points $a_1,a_2 \in \act(\xi)$. 
    Then $d(a_1,\xi) + d(\xi,a_2) \geq d(a_1,a_2) > 2\sqrt[q]{C}$. Therefore one of the distances $d(a_1,\xi), d(\xi,a_2)$ is larger than $\sqrt[q]{C}$ and thus one of the distances $d^q(a_1,\xi), d^q(\xi,a_2)$ is larger than $C$. This contradicts the assumption that both points are in $\act(\xi)$. 

    \item [(ii)] We prove that if $|\act(\xi^\ast)| < n$, then $\xi^\ast$ is not optimal for \eqref{cutoff}. Suppose $|\act(\xi^\ast)| < n$. Then there is a point $a_1 \in \mathcal{A}$ such that $d^q(\xi^\ast, a_1) > C \geq 2^q\cdot \diam(\mathcal{A})^q$ and therefore 
    $$d(\xi^\ast, a_1) > 2\cdot \diam(\mathcal{A}).$$ 
    
    Thus for any $a \in \mathcal{A}$
    $$ d(\xi^\ast,a) \geq \underbrace{d(\xi^\ast,a_1)}_{>2 \diam(\mathcal{A})} - \underbrace{d(a_1,a)}_{\leq \diam(\mathcal{A})} > \diam(\mathcal{A}).$$

    We know now that for all $a\in \mathcal{A}$: $d^q(\xi^\ast,a) \geq \diam(\mathcal{A})^q$. 
    Therefore 
    \\
    $f_C(\xi^\ast) \geq n \cdot \diam(\mathcal{A})^q > (n-1) \cdot \diam(\mathcal{A})^q \geq f_C(a_1)$, which means $\xi^\ast$ is not optimal for \eqref{cutoff}. 
    
    Thus under the conditions of $(ii)$ we do have $|\act(\xi^\ast)|=n$. So we know that $\xi^\ast$ is a barycenter of all points in $\mathcal{A}$ and therefore an optimal solution to \eqref{classic} with $\mathcal{Z}^\ast_C = f_C(\xi^\ast) = f(\xi^\ast) = \mathcal{Z}^\ast$. Since $\xi^\ast \in \mathcal{X}^\ast_C$ was arbitrary, we obtain $\mathcal{X}^\ast_C \subseteq \mathcal{X}^\ast$ and Lemma~\ref{lem:atob} implies $\mathcal{X}^\ast \subseteq \mathcal{X}^\ast_C$, hence $\mathcal{X}^\ast = \mathcal{X}^\ast_C$. 
  \end{itemize}
\end{proof}

\section{Comparing \Nam$_C$ with \Nam$_{C,\alpha}$}\label{sec:emptyCenter}

From an applied point of view it might be interesting to consider the empty barycenter as a valid solution. The barycenter of a set of points is 
representative for said set. Having no barycenter can then be interpreted as "the points are so widely spread, that no single point represents them".

After solving \eqref{cutoff} it is easy to check if the empty barycenter is a better solution. But it would save computation time if we knew before the calculations that the empty barycenter \emph{must} be better. In this section we compare \eqref{cutoff} with \eqref{empty} and work out criteria under which we know that either the empty barycenter is the optimal solution to \eqref{empty} or that the empty barycenter cannot be the optimal solution.
\\[1em]
If the empty barycenter is not the best solution to \eqref{empty} we know that the points of $\mathcal{A}$ must contain a cluster which has a certain \emph{density}. That means that there must exist a subset $A \subseteq \mathcal{A}$ with $\diam(A) \leq 2\sqrt[q]{C}$ containing at least $(1-\alpha)\cdot n$ points:
\begin{lemma}\label{lem:largeActiveXi}
  The empty barycenter is an optimal solution if there is no ball $B$ with radius $\sqrt[q]{C}$ that contains more than $(1-\alpha)\cdot n$ points.
\end{lemma}

\begin{proof}
Suppose such a ball does not exist. Let $\xi^\ast \in \mathbb{R}^k$ be an optimal solution to \eqref{empty}. 
We know for all $a \in \act(\xi^\ast)$ that $d(\xi^\ast,a) \leq \sqrt[q]{C}$. 
Therefore there exists a ball $B$ with radius $\sqrt[q]{C}$ that contains all points of $\act(\xi^\ast)$ and no points of $\con(\xi^\ast)$. 
Since, by assumption, $B$ can not contain more than $(1-\alpha)\cdot n$ points, we know that $\act(\xi^\ast)$ does not contain more than $(1-\alpha)\cdot n$ points, i.e., $\vert \act(\xi^\ast) \vert \leq (1-\alpha)\cdot n$.
But then $\vert \con(\xi^\ast)\vert \geq n - (1-\alpha)\cdot n = \alpha\cdot n$. 
So the points in $\con(\xi^\ast)$ alone contribute at least $\alpha\cdot C \cdot n = f_{C,\alpha}(\emptyset, \mathcal{A})$ to $f_{C,\alpha}(\xi^\ast, \mathcal{A})$, hence $f_{C,\alpha}(\xi^\ast, \mathcal{A}) \geq \alpha\cdot C \cdot n = f_{C,\alpha}(\emptyset, \mathcal{A})$. 
If $\emptyset \notin \mathcal{X}^\ast_{C, \alpha}$, then $f_{C,\alpha}(\xi^\ast,\mathcal{A}) < \alpha\cdot C \cdot n$, which contradicts the optimality of $\xi^\ast$. 
\end{proof}

In Lemma~\ref{lem:largeActiveXi} we could argue with $\act(\xi^\ast)$ alone.  If $\vert \act(\xi^\ast) \vert \leq (1-\alpha)\cdot n$, then the empty barycenter is an optimal solution. 
But we do not know $\xi^\ast$ and therefore $\act(\xi^\ast)$ before solving \eqref{cutoff}. 
Checking if such a ball exists might in general be computationally more easy than solving \eqref{cutoff}.
E.g. for data in $\mathbb{R}^2$ and the Euclidean distance, i.e. $d=\ell_2$, $q=1$, it can be checked in $\mathcal{O}(n^2)$ time if such a ball exists, see \cite{chazelle1986circle}.
\medskip

We further improve Algorithm~\ref{algo:improvement} by using the empty barycenter as an upper bound on the optimal solution to \eqref{empty}. For the empty barycenter we know directly the value $f_{C,\alpha}(\emptyset, \mathcal{A}) = n\cdot \alpha \cdot C$ and initialize the algorithm with this value as current best solution.

\begin{algorithm}[!ht]
  \caption{Second improvement of Algorithm~\ref{algo:Drezner}}\label{algo:improvement2}
  \input{Set $\mathcal{A} = \lbrace a_1, \ldots, a_n \rbrace$, cutoff $C>0$, $\alpha >0$}
  \output{A barycenter $\xi^{\ast}$ of $\eqref{empty}$, objective function value $\mathcal{Z}_{C,\alpha}^\ast$}
  \BlankLine
  Set $\xi^{\ast} \leftarrow \emptyset$, $\mathcal{Z}_{C,\alpha}^\ast \leftarrow n\cdot \alpha \cdot C$, $\mathcal{B} \leftarrow \mathcal{A}$\;
  \For{$i \leftarrow 1$ \KwTo $(n-1)$}{
    $continue$ $\leftarrow$ $true$\;
    $m \leftarrow \vert \lbrace a \in \mathcal{B}\ \vert \ d^q(a_i,a) \leq 2^qC \rbrace \vert$\;
    \If{$(n-m) \cdot C \geq \mathcal{Z}_{C,\alpha}^\ast$}{
      $\mathcal{B} \leftarrow \mathcal{B}\setminus \lbrace a_i \rbrace$\;
      $continue$ $\leftarrow$ $false$\;
    }
    \If{continue}{
      Inner Loop\;
    }
  }
  \Return{$\xi^{\ast}, \mathcal{Z}_{C,\alpha}^\ast$}
\end{algorithm}

\begin{theo}\label{lem:reduction2}
  Let $\mathcal{A} \subseteq\mathbb{R}^2$, let $d$ be a norm-metric and say we can solve \eqref{classic} in $h(n)$ time. Then Algorithm~\ref{algo:improvement2} solves the problem \eqref{cutoff} in $\mathcal{O}(n^2\cdot h(n))$ time.
\end{theo}

\begin{proof}
  The runtime and the correctness of the algorithm follow directly from the proof of Theorem~\ref{lem:reduction1}.
  Formally, the only difference is that Algorithm~\ref{algo:improvement2} is initialized with the empty barycenter as the current best solution.
\end{proof}

\medskip
Compared with Algorithm~\ref{algo:improvement} we replace the initial $\xi^\ast$ in the declaration from $\xi^\ast \leftarrow a_1$ with $\xi^\ast \leftarrow \emptyset$. 
This is of course only better, if $f_{C,\alpha}(\emptyset,\mathcal{A}) \leq f_C(a_1,\mathcal{A})$, which implies $\alpha \leq \frac{n-1}{n}$, compare Lemma~\ref{lem:UpperBounds}.
We will see in the following Lemma that for $\alpha$ close enough to $1$, the empty barycenter cannot be an optimal solution to \eqref{empty}:

\begin{lemma}\label{lem:LowerBoundAlpha}
  If $\alpha > \frac{n -1}{n}$ then the empty barycenter is never an optimal solution.
\end{lemma}

\begin{proof}
  Referring to Lemma~\ref{lem:UpperBounds} we compare $f_{C,\alpha}(\emptyset,\mathcal{A})$, with $\alpha > \frac{n-1}{n}$, to the upper bound of $\mathcal{Z}_{C}^\ast$:
  $f_{C,\alpha}(\emptyset,\mathcal{A}) = \alpha \cdot n \cdot C > \frac{n - 1}{n} \cdot n \cdot C = (n - 1) \cdot C \geq \mathcal{Z}^\ast_{C}$.
  \\
  Hence a point $\xi^\ast \in \mathbb{R}^k$ exists, such that $f_C(\xi^\ast, \mathcal{A}) < f_{C,\alpha}(\emptyset,\mathcal{A})$.
\end{proof}

To determine if the empty barycenter is a better solution than any solution in $\mathbb{R}^k$ before solving $\eqref{cutoff}$, we can look at the pairwise distances between the points of $\mathcal{A}$. 

If the points of $\mathcal{A}$ are close to each other compared to $C$, it is more likely that the cost of an empty barycenter exceeds the cost of a solution in $\mathbb{R}^k$. If on the other hand the points are far apart, it is more likely that the empty barycenter is optimal. We define the \emph{mean pairwise distance} between points of $\mathcal{A}$ and study its relation to the optimal solution of \eqref{empty} w.r.t $n, \alpha$ and $C$.

\begin{de}
  Let $\mathcal{A} = \lbrace a_1, \ldots, a_n \rbrace \subseteq \mathbb{R}^k$. 
  We define the \emph{mean pairwise distance} 
  \begin{align*}
    \mpd(\mathcal{A}) := 
    \frac{1}{n(n-1)}\sum_{i=1}^n\sum_{j=1}^n d^q_{C}(a_i,a_j) = \frac{2}{n(n-1)}\sum_{i=1}^{n-1}\sum_{j=i+1}^n d^q_{C}(a_i,a_j).
  \end{align*}
\end{de} 

The following statements are immediately clear.

\begin{lemma}
  We always have 
  \begin{itemize}
    \item $0 \leq \mpd(\mathcal{A}) \leq C$,
    \item $0 \leq \mpd(\mathcal{A}) \leq \diam(\mathcal{A})^q$.
  \end{itemize}
\end{lemma}

The mean pairwise distance can be computed in $\mathcal{O}(n^2)$ time. If it is `small' compared to the cutoff $C$ and $\alpha$ and $n$, we know that the empty barycenter can again not be an optimal solution to \eqref{empty}. Let us hence study $\mpd_C(\mathcal{A}) := \frac{1}{C}\mpd(\mathcal{A}) \in [0,1]$ as percentage of $C$. We can strengthen Lemma~\ref{lem:LowerBoundAlpha} as follows:

\begin{lemma}\label{lem:emptyNotBetter}
If $\alpha > \mpd_C(\mathcal{A}) \cdot \frac{n-1}{n}$, then for at least one point $a \in \mathcal{A}: f_{C,\alpha}(a,\cA) < f_{C,\alpha}(\emptyset,\cA)$, i.e. the empty barycenter is never an optimal solution.
\end{lemma} 

\begin{proof}
Suppose such a point $a$ does not exist. We show that then $\mpd_C(\mathcal{A}) \cdot \frac{n-1}{n} > \alpha$: For any $i \in \lbrace 1, \ldots ,n \rbrace $: $\sum_{j=1}^n d_C^q(a_i,a_j) = f_{C,\alpha}(a_i,\mathcal{A}) > f_{C,\alpha}(\emptyset,\cA) = \alpha \cdot C\cdot n$. The mean pairwise distance then is
\begin{align*}
  \mpd(\mathcal{A}) &= \frac{1}{n(n-1)}\sum_{i=1}^n\sum_{j=1}^n d^q_{C}(a_i,a_j)
  \\
  &> \frac{1}{n(n-1)} \sum_{i=1}^n \alpha \cdot C \cdot n =\frac{1}{n(n-1)}\cdot \alpha \cdot C \cdot n^2  = \alpha \cdot C \cdot\frac{n}{n-1},
\end{align*}
hence $\mpd_C(\mathcal{A}) \cdot \frac{n-1}{n} > \alpha$.
\end{proof}

We can directly transfer this result to the diameter $\diam(\mathcal{A})$ which we used in Section~\ref{sec:cutoff}, since the mean pairwise distance is never larger than the diameter raised to the power $q$.

\begin{corollary}
  If $\alpha > \frac{\diam(\mathcal{A})^q}{C} \cdot \frac{n-1}{n}$ then the empty barycenter is never an optimal solution. 
\end{corollary}
\begin{proof}
  From $\alpha > \frac{\diam(\mathcal{A})^q}{C} \cdot \frac{n-1}{n}$ it follows that $\alpha > \mpd_C \cdot \frac{n-1}{n}$. With Lemma~\ref{lem:emptyNotBetter} we know that then the emtpy barycenter is not an optimal solution to \eqref{empty}.
\end{proof}

We can reformulate Lemma~\ref{lem:emptyNotBetter} and Corollary~$25$ to get a condition for the diameter and the mean pairwise distance for the empty barycenter not being optimal:
$\mpd(\mathcal{A}) < \alpha \cdot C \cdot \frac{n}{n-1}$ or $\diam(\mathcal{A})^q < \alpha\cdot C \cdot \frac{n}{n-1}$ then $\emptyset$ is not optimal.

On the other hand we show that for small $\alpha$ the empty barycenter is always an optimal solution to \eqref{empty}:
\begin{lemma}\label{lem:smallAlpha}
  If $\alpha \leq \min \left\{ \frac{1}{2^q} \frac{\diam(\mathcal{A})^q}{n\cdot C}, \frac{1}{n} \right\}$ then the empty barycenter is an optimal solution to \eqref{empty}.
\end{lemma}

\begin{proof}
  Let $\xi^\ast \in \mathcal{X}^\ast_C$ be an optimal solution to \eqref{cutoff}. 
  We know from the triangle inequality that there exists a point $a \in \mathcal{A}$ such that $d(\xi^\ast,a) \geq \frac{1}{2} \diam(\mathcal{A})$. 
  Therefore $d_C^q(\xi^\ast,a) \geq \min \lbrace \frac{1}{2^q} \diam(\mathcal{A})^q, C \rbrace$ and hence $\mathcal{Z}_C^\ast \geq \min \lbrace \frac{1}{2^q} \diam(\mathcal{A})^q, C \rbrace$.
  Then 
  \begin{align*}
    f_{C,\alpha}(\emptyset,\mathcal{A}) &= n\cdot \alpha \cdot C \leq \min \left\{\frac{1}{2^q} \diam(\mathcal{A})^q, C \right\}
    \\
    \Leftrightarrow \alpha &\leq \min \left\{ \frac{1}{2^q} \frac{\diam(\mathcal{A})^q}{n\cdot C}, \frac{1}{n} \right\}
  \end{align*}  
\end{proof}

\begin{remark}
  When $\diam(\mathcal{A}) \geq 2C$ then $\min \left\{ \frac{1}{2^q} \frac{\diam(\mathcal{A})^q}{n\cdot C}, \frac{1}{n} \right\} = \frac{1}{n}$. And thus for $\alpha \leq \frac{1}{n}$ and $\diam(\mathcal{A}) \geq 2C$ the empty barycenter is an optimal solution to \eqref{empty}.
\end{remark}

The following lemma and example show that for larger $\mpd(\mathcal{A})$ both cases, $\emptyset \in \mathcal{X}^\ast_{C,\alpha}$ or $\emptyset \not\in \mathcal{X}^\ast_{C,\alpha}$, are possible.

\begin{lemma}\label{lem:ChalfplusEps}
Let $\alpha=\frac{1}{2}, q=1$. Then for any $\varepsilon > 0$ there exists a set $\mathcal{A} \subseteq \mathbb{R}^k$ such that $\mpd(\mathcal{A}) = \frac{1}{2} \cdot C + \varepsilon$, but where $f_{C,\frac{1}{2}}(\emptyset,\mathcal{A}) < f_{C,\frac{1}{2}}(x,\mathcal{A})$ for any $x \in \mathbb{R}^k$.
\end{lemma}

\begin{proof}
We construct a set $\mathcal{A} = \lbrace a_1, \ldots, a_{2n} \rbrace \subseteq \mathbb{R}$, where $n > \frac{C}{4\varepsilon} + \frac{1}{2}$ and set $\delta := \frac{2n(2n-1)\varepsilon - nC}{4(n-1)}$. 
The points have the coordinates $a_1=-\delta, a_2=\ldots =a_n=0, a_{n+1}= \ldots = a_{2n-1} = C, a_{2n}=C+\delta$.
Then $\mpd(\mathcal{A}) = 
\frac{n^2C + 2(n-1)\delta}{n(2n-1)} = \frac{nC}{2n-1} + \frac{2(n-1)\delta}{n(2n-1)} = \frac{C}{2} + \frac{C}{2(2n-1)} + \frac{2(n-1)\delta}{n(2n-1)}$.With the specified $\delta$ we have $\mpd(\mathcal{A}) = \frac{C}{2} + \varepsilon$ and since $n  > \frac{C}{4\varepsilon} + \frac{1}{2}$ we have $\delta > 0$.

An optimal solution of \eqref{cutoff} is any of the points $a_2, \ldots a_{2n-1}$. The optimal objective function value $\mathcal{Z}^\ast_C = f_{C,\frac{1}{2}}(a_2,\mathcal{A}) = nC + \delta$ is larger than $nC = f_{C,\frac{1}{2}}(\emptyset,\mathcal{A})$. So the empty barycenter is a better solution than the best solution in $\mathbb{R}$.

\end{proof}
We have seen in Lemma~\ref{lem:emptyNotBetter} that for $\mpd(\mathcal{A}) < \alpha \cdot C \cdot \frac{n}{n-1}$ the empty barycenter is not optimal. 
Yet, Lemma~\ref{lem:ChalfplusEps} proves that the $\mpd$ can be arbitrarily close to $\alpha \cdot C$ and still the empty barycenter \emph{is} an optimal solution.
The following example proves on the other hand that there exist sets with an $\mpd$ arbitrarily close to $C$, for which the empty barycenter is \emph{not} an optimal solution.

\begin{example}\label{ex:CminusEps}[Example for large $\mpd(\mathcal{A})$ where $\emptyset \not\in \mathcal{X}^\ast_{C,\alpha}$ for $\alpha \geq \frac{1}{2}$.] 

Let two points in $\mathbb{R}$ be given, $a_1=0, a_2=\sqrt[q]{C-\varepsilon}$ for some $\varepsilon > 0$. Now $f_{C,\alpha}(a_1,\mathcal{A}) = C - \varepsilon < C \leq 2\cdot \alpha \cdot C = f_{C,\alpha}(\emptyset,\mathcal{A})$.
\end{example}
Note that the situation of this example is the same if we replace the $\mpd$ by the \emph{minimum} or \emph{median} distance between points, since both values are $C-\varepsilon$.

We finally summarize our findings.
We know that the empty barycenter \emph{is} an optimal solution to \eqref{empty} if
\begin{itemize}
  \item there is no ball $B$ with radius $\sqrt[q]{C}$ that contains at least $(1-\alpha)\cdot n$ points, see Lemma~\ref{lem:largeActiveXi},
  \item $\alpha \leq \min \left\{ \frac{1}{2^q} \frac{\diam(\mathcal{A})^q}{n\cdot C}, \frac{1}{n} \right\} $, see Lemma~\ref{lem:smallAlpha}.
\end{itemize}

On the other hand we know that the empty barycenter \emph{is not} an optimal solution if
\begin{itemize}
  \item $\alpha > \frac{n-1}{n}$, see Lemma~\ref{lem:LowerBoundAlpha},
  \item $\alpha > \mpd_C(\mathcal{A}) \cdot \frac{n-1}{n}$ or $\alpha > \frac{\diam(\mathcal{A})^q}{C} \cdot \frac{n-1}{n}$, see Lemma~\ref{lem:emptyNotBetter} and the subsequent Corollary.
\end{itemize}
In the other cases 
we do not know in advance if the empty barycenter is optimal, c.f. Lemma~\ref{lem:ChalfplusEps} and Example~\ref{ex:CminusEps}.

\section{Sensitivity analysis w.r.t $C$} \label{sec:variableC}

So far we have assumed that the cutoff value $C$ is a priori specified, but there is a wide range of scenarios where this is not the case.

In the application described in Section~\ref{sec:applications} we often (but not always) know the order of magnitude of a reasonable cutoff $C$ due to the physical reality of the data, but this typically still leaves a large interval of possible choices which may lead to very different outcomes.

In a more direct location problem setting the actual $C$ may be determined by another player trying to maximize her profit based on knowledge of the entire function $g = [C \mapsto  \min_{\xi} f_C(\xi,\mathcal{A})]$ (or $[C \mapsto  \min_{\xi} f_{C,\alpha}(\xi,\mathcal{A})]$, where we still assume $\alpha>0$ to be fixed). 
Taking up the waste dump example from the introduction, it may be that in the decision process the local transportation company is asked for the price $C$, at which it would offer to transport domestic waste to the dump. A profit maximizing choice of $C$ depends on detailed knowledge of the function $g$.

This function is what we study in the present section.

\begin{de}
  Let $x \in \mathbb{R}^k$ be a fixed point. Define the two functions
  \begin{itemize}
    \item [(i)] $g_{x}: \mathbb{R}_{+} \rightarrow \mathbb{R}, C \mapsto  f_C(x,\mathcal{A})$
    \item [(ii)] $g: \mathbb{R}_{+} \rightarrow \mathbb{R}, C \mapsto  f_C(\mathcal{A}) := \underset{\xi \in \mathbb{R}^k}{\min}\ f_C(\xi,\mathcal{A})$.
  \end{itemize}
  The function $g_{x}$ maps the cutoff $C$ to the objective function value $f_C(x,\mathcal{A})$, while the function $g$ maps the cutoff to the \emph{optimal} objective function value $\mathcal{Z}_C^\ast$, compare with \eqref{cutoff}.
\end{de}

We study how the barycenter $\xi^\ast = \xi^{\ast}_{C}$ and the values of $g_x$ and $g$ change with changing $C$.

\begin{example}\label{ex:discontinuity}[Example for discontinuity of the optimal solution to \eqref{cutoff} w.r.t $C$]
\begin{center}
  \begin{tikzpicture}[scale = 1]
    \draw [draw=black] (-0.6,-0.6) rectangle ++(8.8,1.2);
    \tkzInit[xmax=8,ymax=0.5,xmin=-0.5,ymin=-0.5]
    \draw [->, color=gray, dotted] (-0.4,0) -- (8,0);
    \draw [-, color=gray, dotted] (0,-0.2) -- (0,0.2);
    \node  (1) at (0, 0) {$a_1$};
    \node  (2) at (0.5, 0) {$a_2$};
    \node  (3) at (5, 0) {$a_3$};
    \node  (4) at (6, 0) {$a_4$};
    \node  (5) at (7, 0) {$a_5$};
  \end{tikzpicture}
\end{center}
  
  Let $5$ points in $\mathbb{R}$ be given, $a_1=0,a_2=0.5,a_3=5,a_4=6,a_5=7$. Consider $q=1$ and $d = |\cdot |$. The location of the points is sketched above. 
  The following gives a complete description of $\mathcal{X}^\ast_C$ for various $C$. Note that at the boundaries of the ranges the union of the barycenters in the lower and the higher range are in $\mathcal{X}^\ast_C$.
  For $0 \leq C \leq 0.5$ any of the points $a_1$ to $a_5$ is optimal for \eqref{cutoff}.
  For $0.5 \leq C \leq 1.5$ any point on the interval $[a_1,a_2]$ is optimal, 
  for $1.5 \leq C \leq 5.25$ $a_4$ is optimal, 
  for $5.25 \leq C$ $a_3$, the barycenter of the problem without cutoff, is optimal. 
\end{example}

As seen in the example, the optimal solution set for \eqref{cutoff} can change abruptly in $C$. 
On the other hand we will see that $g$, i.e. the objective function $C \mapsto f_C(\mathcal{A})$ is quite well-behaved and can be computed efficiently.

\begin{lemma}\label{lem:locallyConcaveAndCont}
Let $x \in \mathbb{R}^k$ be some fixed point, $\mathcal{A} = \lbrace a_1, \ldots, a_n \rbrace \subseteq \mathbb{R}^k$. 
Sort the points in\-crea\-sing\-ly by their distance to $x$ and define $d_i := d^q(x,a_i)$, so that $d_1 \leq d_2 \leq \ldots \leq d_n$. 
Then the function $g_x$ is 
  \begin{itemize}
    \item[(i)] of the form \begin{equation}
      \label{eq:gxofC}
      g_x(C) = \sum_{i = 1}^j d_i + (n-j)\cdot C, \ \text{ with } j=|\act_C(x)| 
    \end{equation}
    $g_{x}$ is therefore piecewise linear with kinks in $d_i$.
    \item[(ii)] continuous,
    \item[(iii)] non-decreasing, 
    \item[(iv)] concave.
  \end{itemize}
\end{lemma}

\begin{proof} (i) 
W.l.o.g we assume that $x \neq a_i$ for all $i \in \lbrace 1,2,\ldots,n\rbrace$. Otherwise we eliminate the first $j$ points from our list, where $j:= \max \lbrace i \in \lbrace 1, \ldots, n \rbrace \ \vert \ d_i =0 \rbrace$, use $n^\prime = n-j$ in the proof and re-enumerate $a_{j+1},\ldots,a_n$ to $a_1,\ldots,a_{n^\prime}$.
\\[0.5em]
By definition
\begin{align*}
g_x(C) = f_C(x,\mathcal{A}) =  \sum_{i=1}^n \text{min}\lbrace d_i,C \rbrace.
\end{align*}
We can split this sum into sums over indices of active and constant points, as defined in Section~\ref{sec:LocalStructure}. 
Let $m = m_C = |\act_C(x)|$. Then $d_1 \leq \ldots \leq d_{m_C} \leq C$ and $C < d_{m_C+1} \leq \ldots \leq d_n$ and 
\begin{align*}
  g_x(C) = \sum_{i=1}^n \text{min}\lbrace d_i,C \rbrace = \sum_{i=1}^{m_C} d_i + \sum_{i=m_C+1}^n C = \sum_{i=1}^{m_C} d_i + (n - m_C)\cdot C .
\end{align*}
For any $j \in \lbrace 1, \ldots , n \rbrace$ the sum over the $d_i$ is constant for $d_{j} \leq C < d_{j+1}$ so $g_x$ is piecewise linear. 
The slope of the $j$-th line segments is $(n-j)$.
\\[0.5em]
(ii) We know that $g_x$ is piecewise linear with kinks in $d_i$. On these line segments, i.e. $d_i < C < d_{i+1}$, the function is  continuous. We have to check for the kinks of $g_x$, i.e. $C=d_j$ for some $j$, if the two line segments for $C = d_j - \varepsilon$ and $C = d_j + \varepsilon$, $\varepsilon > 0$, intersect at $g_x(d_j)$.
\medskip

Take $j \in \lbrace 1, \ldots , n \rbrace$. For $d_{j-1} < C < d_j$: $g_x(C) = \sum_{i=1}^{j-1} d_i + \sum_{i=j}^n C $. For $j=1$ the first sum is $0$.
For $d_{j} < C < d_{j+1}$: $g_x(C) = \sum_{i=1}^{j} d_i + \sum_{i=j+1}^n C $. For $j=n$ the second sum is $0$.

For $C=d_j$: $\sum_{i=1}^{j-1} d_i + \sum_{i=j}^n C = \sum_{i=1}^{j-1} d_i + d_j + \sum_{i=j+1}^n C = \sum_{i=1}^{j} d_i + \sum_{i=j+1}^n C$. So the two line segments intersect and therefore is $g_x$ continuous. 
\\[0.5em]
(iii) Take two cutoffs $C_1 < C_2$. Since $C_1 < C_2$ we also have $\text{min}\lbrace d_i,C_1 \rbrace \leq \text{min}\lbrace d_i,C_2 \rbrace$ for all $i \in \lbrace 1,\ldots ,n \rbrace$. Now $g_x(C_1) = \sum_{i=1}^n \text{min}\lbrace d_i,C_1 \rbrace \leq \sum_{i=1}^n \text{min}\lbrace d_i,C_2 \rbrace = g_x(C_2)$, so $g_x$ is non-decreasing.
\\[0.5em]
(iv) With larger $C$ the cardinality $m_C$ of $\act_C(x)$ increases. 
Therefore the slope $(n-m_C)$ of the line segments decreases with growing $C$, so $g_x$ is also concave.
\end{proof}

We can now extend the results for $g_x$ which hold for all $x \in \mathbb{R}^k$ to the function $g$.

\begin{theo} \label{thm:gContNondecConc}
  The function $g$ is continuous, non-decreasing and concave.
\end{theo}
  
\begin{proof}
  For calculating the function $g$ we need to solve \eqref{cutoff} for every $C$. For a fixed $C$ an optimal solution to $\eqref{cutoff}$ is a solution of \eqref{classic} for some subset $A \subseteq \mathcal{A}$, compare Lemma~\ref{lemma1}.
  
  Since $\mathcal{A}$ is finite, there are only finitely many subsets of $\mathcal{A}$. 
  There is therefore only a finite set $S$ of candidates for an optimal solution to \eqref{cutoff}. 
  The function $g$ is the minimum of the functions $g_{\xi}$, i.e. $g(C) = \min_{\xi \in S} g_{\xi}(C)$. 
  
  The minimum of finitely many continuous functions is continuous. The same holds for the properties ``non-decreasing'' and ``concave''.
\end{proof}

Recall the piecewise linear form of the function $g_x$. 
The function $g$ as minimium of finitely many piecewise linear functions is then itself piecewise linear. 
The slopes of the line segments are given by the cardinalities of the sets $\act_{C}(\xi^\ast)$ for optimal solutions $\xi^\ast = \xi^\ast_C \in \mathcal{X}^\ast_C$ for the different $C$. Those slopes are integers between $n-1$ and $0$. Since $g$ is continuous and concave the function  
consists of at most $n$ linear pieces. The slope only changes at the kinks of the function $g$. 
And only there does the cardinality of the set $\act(\xi^\ast)$ change. 
Let us say we have kinks at $C_1 < C_2$ and no kinks in between. 
Let $C \in (C_1, C_2)$. Any solution $\xi^\ast_{C} \in \mathcal{X}^\ast_{C}$ defines the same function $g_{\xi^\ast_{C}}$ on the interval $(C_1,C_2)$. 
The function $g$ has its next kink at $C_2$, 
so $C_2$ is the smallest value greater than $C_1$ where any of the functions $g_{\xi^\ast_{C}}$ can have a kink.
It follows that any solution $\xi^\ast_{C} \in \mathcal{X}^\ast_{C}$ is optimal on the whole interval $[C_1, C_2]$.
\\
But that means if we find all values $C$ at which $g$ has a kink, and a corresponding optimal solution $\xi^\ast$ for each of those $C$, we have an optimal solution to \eqref{cutoff} and the value of $g$ for any $0 \leq C < \infty$.

\medskip
We describe in Algorithm~\ref{algo:calAllC} how we can calculate these optimal solutions $\xi^\ast$ and values of $g$ in at most $n-1$ steps, by finding the different line segments.
The function $\mathtt{bar}(C,\mathcal{A})$ calculates an optimal solution $\xi^\ast$ of \eqref{cutoff} and the corresponding value $\mathcal{Z}^\ast_C$ for a given cutoff $C$. $\mathtt{bar}(\infty,\mathcal{A})$ returns an optimal solution to \eqref{classic}.
\\
Let the set $\mathcal{S} \subseteq \lbrace 0, \ldots ,n-1 \rbrace$ contain the \emph{slopes} of the line segments that we have already found and the set $\mathfrak{O} \subseteq \lbrace 0, \ldots ,n-1 \rbrace \setminus \mathcal{S}$ the slopes of the segments that we still might find, the \emph{open} slopes. 
Each of the lines $l_{n-1}, \ldots, l_0$ is 
defined by a point $(C,\mathcal{Z}^\ast_C)$ and the slope $i$, which is indicated by its index. 
The algorithm will calculate (up to) $n$ different lines. These lines are tangents for the function $g$. 
By calculating the intersection points of lines $l_i$ and $l_{i+1}$ we get the kinks of $g$ and thereby the complete function $g(C)$, which is a combination of segments of the lines $l_{n-1}, \ldots, l_0$.

\SetKwInOut{input}{Input}\SetKwInOut{output}{Output}
\begin{algorithm}[!htb]
  \caption{Calculate $f_C(\mathcal{A})$ for all $C$}\label{algo:calAllC}  
  \input{The set $\mathcal{A} = \lbrace a_1 ,\ldots, a_n \rbrace$}   
  \output{The function $g(C)$ and up to $n-1$ barycenters corresponding to the different values of $C$.}
  \BlankLine
  Set $\mathcal{S} \leftarrow \lbrace 0,n-1 \rbrace$, $\mathfrak{O} \leftarrow \lbrace 1,\ldots,n-2 \rbrace$\;
  $(\xi^\ast_0,\mathcal{Z}^\ast_0) \leftarrow \mathtt{bar}(\infty,\mathcal{A})$, $(\xi^\ast_{n-1},\mathcal{Z}^\ast_{n-1}) \leftarrow (a_1,0)$\;
  Define the lines $l_0$ by the point $(0,\mathcal{Z}^\ast_0)$ and the slope $0$ and $l_{n-1}$ by the point $(0,0)$ and the slope $n-1$ \;
  \While{$\mathfrak{O} \neq \emptyset$}{
    Take smallest index $o$ from $\mathfrak{O}$\;
    Take the largest $i \in \mathcal{S}: i<o$ and the smallest $j \in \mathcal{S}: o<j$\;
    Calculate the intersection point $(C,y)$ of lines $l_i$ and $l_j$\;
    $(\xi^\ast_C, \mathcal{Z}^\ast_C) \leftarrow \mathtt{bar}(C,\mathcal{A})$ \;
    \If{$\mathcal{Z}_C^\ast = y$}{
      $\mathfrak{O} \leftarrow \mathfrak{O}\setminus\lbrace i+1, i+2, \ldots,j-1 \rbrace$\;
    }
    \Else{
      Set $m \leftarrow |\act(\xi^\ast_C)|$, $\xi^\ast_{n-m} \leftarrow \xi^\ast_C$, $\mathcal{Z}^\ast_{n-m} \leftarrow \mathcal{Z}_C^\ast$\;
      Define $l_{n-m}$ by the point $(C,\mathcal{Z}^\ast_{n-m})$ and slope $n-m$\;
      $\mathfrak{O} \leftarrow \mathfrak{O}\setminus\lbrace n-m \rbrace$\;
      $\mathcal{S} \leftarrow \mathcal{S}\cup\lbrace n-m \rbrace$\;
    }
  }
  \For{$i\in \mathcal{S} \setminus \lbrace n-1 \rbrace$}{
    $j \leftarrow \min \lbrace k\in \mathcal{S} \ \vert \ k > i \rbrace$\;
    Calculate the intersection point $(C_i,g(C_i))$ of lines $l_i$ and $l_j$\;
  }
  $(C_{n-1},g(C_{n-1})) \leftarrow (0,0)$\;
  \Return{$\lbrace (C_i,g(C_i)) \ \vert \ i\in \mathcal{S} \rbrace$, $\lbrace \xi^\ast_i | i \in \mathcal{S} \rbrace$}
\end{algorithm}

\begin{theo}
  Say we can solve \eqref{cutoff} in $h_C(n)$ time. 
  Then Algorithm~\ref{algo:calAllC} computes the function $g$ in $\mathcal{O}(n\cdot h_C(n))$ time.
\end{theo}

\begin{proof}
  We have to prove two things: first the runtime and second the correctness.
  \\
  First: The function $\mathtt{bar}(C,\mathcal{A})$ is called at most $n-1$ times. Once in line $2$ and once in every iteration of the while-loop, lines $4$ to $18$. The while-loop is called at most $n-2$ times.
  The computation of the (at most) $n-1$ intersection points in line $19$ to $22$ is done in $\mathcal{O}(n)$ time.
  Together we have a runtime of $\mathcal{O}(n\cdot h_C(n) + n) = \mathcal{O}(n\cdot h_C(n))$.
  \\
  Second:
  We know by Lemma~\ref{lem:locallyConcaveAndCont} that $g$ starts in $(0,0)$ with slope $n-1$ and will eventually get constant, taking the value $g(C) = \mathcal{Z}^\ast_0 = \mathtt{bar}(\infty,\mathcal{A})$. 
  So the lines $l_0$ and $l_{n-1}$ defined in line $3$ contain the outermost line segments of $g$ and bound $g$ above due to its concavity.
  \\
  We have to prove that the algorithm finds all line segments in between.
  Suppose we have two line segments $l_i$ and $l_j$.
  Now we compute the intersection of those lines in line $7$. 
  We know that this intersection point $(C,y)$ must lie on or above $g$ by concavity. We now calculate the value $\mathcal{Z}_C^\ast$ for this $C$ and get a point $(C,\mathcal{Z}_C^\ast)$ that is on $g$. There are two possibilities:
  \begin{itemize}
    \item Either $y = \mathcal{Z}_C^\ast$, which means the intersection point of the lines \emph{is} already on $g$. But that means that the function $g$ can have no line segments with slopes between $i$ and $j$, so all these values are removed from the set $\mathfrak{O}$ in line $10$.
    \item Or we found a point on $g$ that is below the intersection point $(C,y)$. For the found barycenter $\xi^\ast$ we look at the cardinality $m$ of $\act(\xi^\ast)$. The slope of the line $l_{n-m}$, which 
  is a tangent on $g$, is given by the \emph{non}-active points. So the next line $l_{n-m}$ is given by the slope $n-m$ and the point $(C,\mathcal{Z}_c^\ast)$. This line is saved and we add $n-m$ to $\mathcal{S}$ and delete $n-m$ from $\mathfrak{O}$.
  \end{itemize} 
  So with every iteration of lines $4$ to $18$ we either find out that $l_i$ and $l_j$ intersect on $g$ or we find one new line segment that is a tangent for $g$. The function $g$ is uniquely defined by these tangents.
\end{proof}

 Note that Algorithm~\ref{algo:calAllC} can be parallelized. When the intersection of two lines is calculated in line $7$, 
   the problem can then be split into a subproblem to the left of this point and to the right of this point.  

\begin{remark}
  Having computed $g$, the function $g^{(\alpha)} : \mathbb{R}_{+} \to \mathbb{R}$, $C \mapsto \min_{\xi \in \mathbb{R}^k \cup \{\emptyset\}} f_{C,\alpha}(\xi,\mca)$ is easily derived, since $g^{(\alpha)}(C) = \min \{ g(C), \alpha \cdot n \cdot C \}$. Provided that $\alpha < \frac{n-1}{n}$, which means $\alpha \cdot n$ is smaller than the initial slope of $g$, we obtain from the concavity of $g$ and the fact that it must eventually be constant, that there is exactly one $C_0>0$ where the graph of $g$ intersects with the linear function $[C \mapsto \alpha \cdot n \cdot C]$. We then have $g^{(\alpha)}(C) = \alpha \cdot n \cdot C$ to the left of $C_0$ and $g^{(\alpha)}(C) = g(C)$ to the right of $C_0$. If $\alpha \geq \frac{n-1}{n}$, we have $g^{(\alpha)} = g$ everywhere.
\end{remark}

\section{Applications}\label{sec:applications}

The original motivation for investigating \eqref{empty} comes from \cite{muller2020metrics}, where two of the current authors studied barycenters of finite collections of point patterns for their use as summary statistics. We briefly describe here the relevant details, because we think that the involved concepts and their algorithmic implications may well be of interest in the context of location problems where e.g.\ an optimal supply chain is to be maintained to a number of companies that each have several branch offices. 

For the present purpose we define a point pattern as a finite subset of $\R^k$ and denote the set of all such patterns by $\mfn$. Then for given point patterns $\xi_1,\ldots,\xi_m \in \mfn$, a barycenter is any minimizer of the Fr\'echet functional
\begin{equation} \label{eq:frechet}
  F(\zeta) = \sum_{j=1}^m \tau(\xi_j,\zeta)^q
\end{equation}
over $\zeta \in \mfn$. Here $\tau$ is the transport-transform (TT) metric on $\mfn$ introduced in \cite{muller2020metrics}. Basically, $\tau(\xi_j,\zeta)^q$ is the minimal ``cost'' of matching a subset of $\xi_j$ and a subset of $\zeta$, where each pairing of a point $x \in \xi_j$ and a point $z \in \zeta$ incurs a cost of $d(x,z)^q$ and each unmatched point of either pattern incurs a cost of $\frac12 C$.

If we consider point patterns as discrete measures by identifying $\xi = \{x_1,\ldots,x_n\}$ with $\sum_{i=1}^n \delta_{x_i}$ for pairwise distinct $x_i$, we can re-interpret the TT metric as a special case of an unbalanced Wasserstein metric, see \cite{chizat2018scaling} for the definition of the latter or \cite{muller2020metrics}, Remark~3, for the full argument.

Intuitively, a barycenter can be thought of as a ``typical'' representative, in a sense an ``average point pattern'' that reflects common properties of the data point patterns. In \cite{muller2020metrics} barycenters were applied to point patterns of crime locations in two cities, with the goal of detecting systematic differences over the years or between different seasons. Another goal might be for planning the efficient deployment of police officers according to the time of the day (or year) and maybe other side constraints (predictive policing).

\cite{borgwardt2021computational} prove that the computation of a sparse Wasserstein barycenter is $\mathcal{NP}$-hard for three point patterns with the same number of points in $\mathbb{R}^2$ and $q=2$. In the authors' setting the barycenter can be a more general discrete finite measure (not necessarily with unit weights), but their sparseness condition limits the number of support points.
There does not seem to be a direct theoretical result for our problem \eqref{eq:frechet}, but based on the current state of theoretical and applied research, we assume that this problem is insolvable for all practical purposes. Therefore \cite{muller2020metrics} proposed a heuristic algorithm based on an equivalent form of the TT metric: First fill up the point patterns $\xi_1, \ldots, \xi_m$ so that they all have the same cardinality $n$, say, by adding points at a single ``virtual'' location $\aleph \not\in \R^k$ at distance $(\frac12 C)^{1/q}$ apart from any locations in $\R^k$. For $\xi_j = \{x_{j1},\ldots,x_{jn}\}$ and $\zeta = \{z_{1},\ldots,z_{n}\}$ (multisets since they may include $\aleph$ several times), we may then express the metric $\tau$ equivalently as

\begin{equation}  \label{eq:ourproblem}
  \tau(\xi_j,\zeta)^q = \min_{\pi \in S_n} \sum_{i=1}^n d'(x_{ji},z_{\pi(i)})^q,
\end{equation}
where $S_n$ denotes the set of permutations on $\{1,\ldots,n\}$ and
\begin{equation}  \label{eq:dprime}
  d'(x,z)^q = \begin{cases}
               \min \bigl\{d(x,z)^q, C\} &\text{if $x,z \in \R^k$};\\
               \frac12 C                   &\text{if $\aleph \in \{x,z\}$, $x \neq z$};\\
               0                       &\text{if $x=z=\aleph$};
            \end{cases}
\end{equation}
see \cite{muller2020metrics}, Theorem~1. We may then find a local optimum of the Fr\'echet functional~\eqref{eq:frechet} by alternating between forming pairwise disjoint clusters of the form $\mathcal{C} = \{x_{1,i_1}, \ldots, x_{m,i_m}\}$, $i_1, \ldots, i_m \in \{1,\ldots, n\}$, including exactly one (maybe virtual) point from each data pattern via optimal matching, and computing suitable ``centers'' for each such cluster $\mathcal{C}$ by minimizing
\begin{equation}  \label{eq:the_problem_prime}
  f(z) = \sum_{j=1}^m d'(x_{j,i_j},z)^q
\end{equation}
over $z \in \R^k \cup \{\aleph\}$. In the algorithm of \cite{muller2020metrics} this minimization was only performed approximately, using some crude but fast heuristics. However, except for the fact that $x_{j,i_j} = \aleph$ may hold for individual $j$, the minimization \eqref{eq:the_problem_prime} corresponds to problem (\ref{empty}) with $\alpha = \frac12$. Noting that the contribution from $x_{j,i_j} = \aleph$ is constant as long as $z \in \R^k$, we may therefore use a slightly adapted version of Algorithm~\ref{algo:improvement2} to compute the centers exactly.

\section{Simulation study}\label{sec:Simulation}

For comparing Drezners algorithm with the two improvements Algorithm~\ref{algo:improvement} and Algorithm~\ref{algo:improvement2}
, we created six test scenarios of point patterns inside the unit square and compared runtimes and solutions of the algorithm. 
For scenarios (1) to (5) we chose rectangles and generated the coordinates of the points inside each rectangle uniformly at random, independently of one another. 
In scenario (2) to (5) we combined two of those rectangles. 
The number of points in every rectangle follows a Poisson distribution with parameters chosen in such a way that the expected number of points is $600$ in each scenario. 
The scenarios are (from left to right, top to bottom):

\begin{minipage}{0.5\textwidth}
  \begin{itemize}
    \item [(1)] one unit square \vspace*{-0.6em}
    \item [(2)] two squares with edge length $0.5$ that overlap in a square of size $0.1 \times 0.1$, half of the points in each square \vspace*{-0.6em}
    \item [(3)] one small square with edge length $0.4$ inside the unit square, half of the points in each square \vspace*{-0.6em}
    \item [(4)] one small square with edge length $0.3$ inside the unit square, half of the points in each square \vspace*{-0.6em}
    \item [(5)] two rectangles overlapping on one strip of width $0.2$. Height for both rectangles is $0.5$, width $0.5$ and $0.6$, half of the points in each rectangle \vspace*{-0.6em}
    \item [(6)] $4$ small clusters with background noise. The clusters are two-dimensional Gaussians with $\sigma = 0.025$. The clus\-ter cen\-ters are uniformly drawn for each pattern individually. The expected num\-ber of points in the clusters is $90\%$, an expected number of $10\%$ are uniformly drawn in the unit square. 
  \end{itemize}
\end{minipage}
\begin{minipage}{0.5\textwidth}
  \vspace*{-1.3em}
  \includegraphics[scale=0.25]{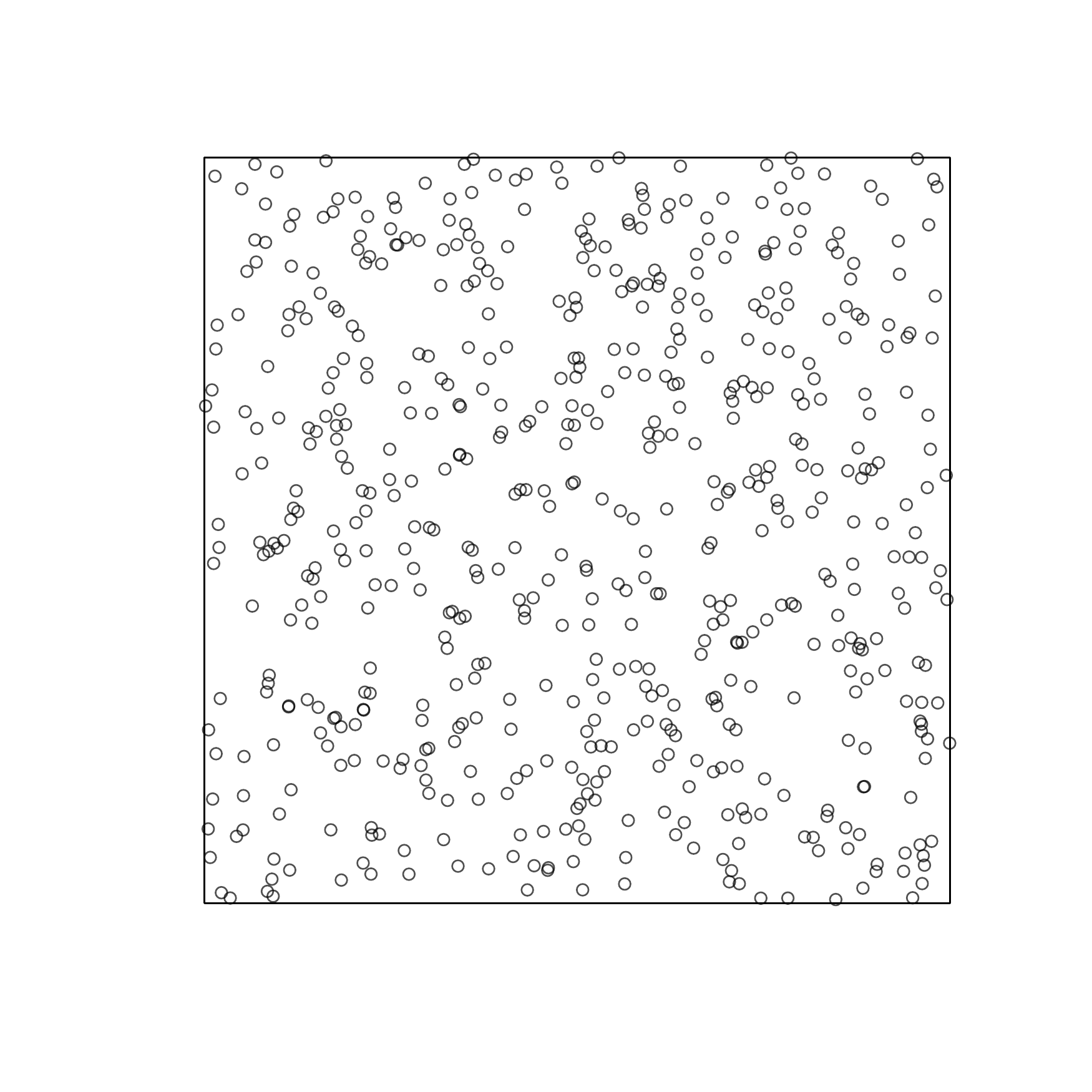}\hspace*{-1.73em}
  \includegraphics[scale=0.25]{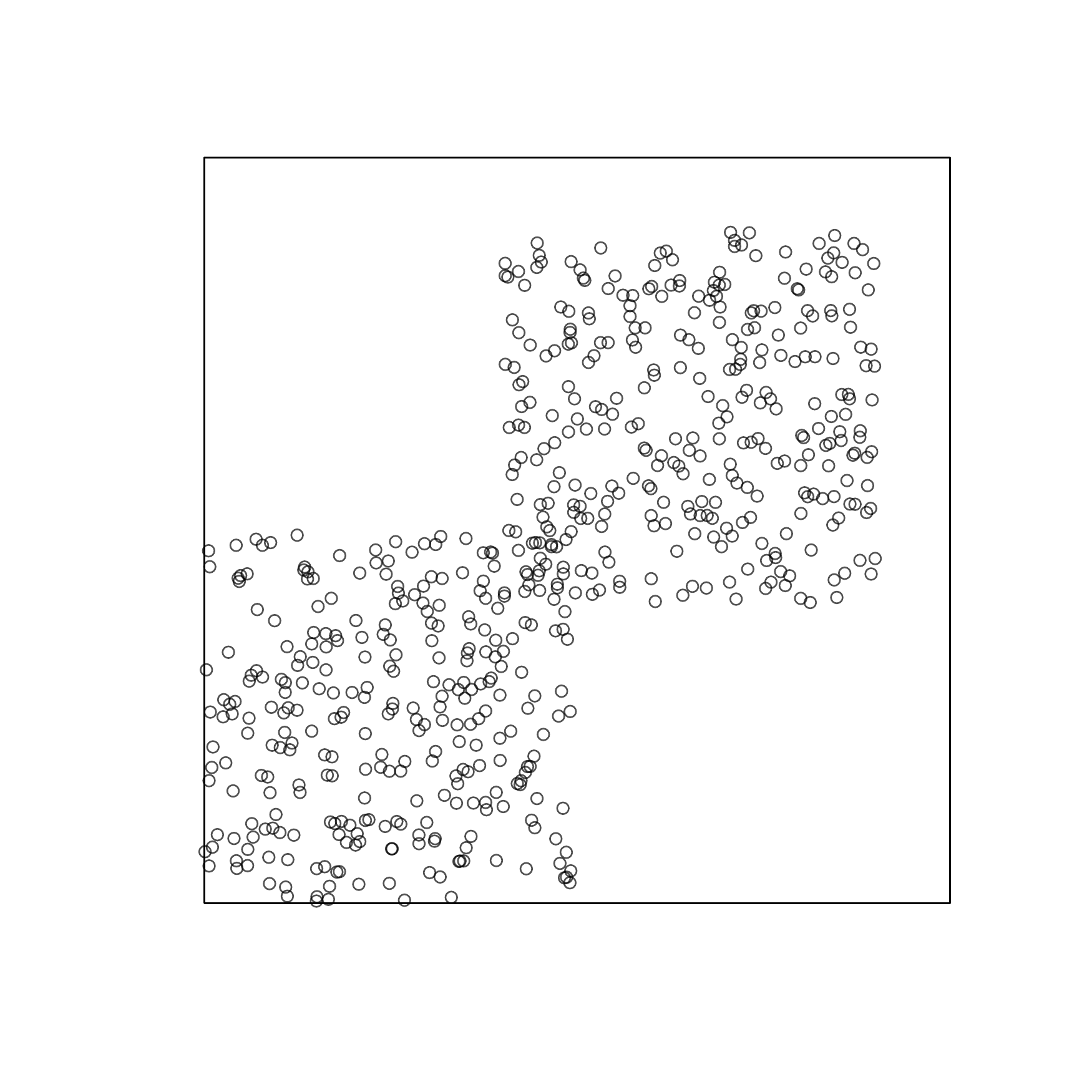}\vspace*{-1.99em}
  \\
  \includegraphics[scale=0.25]{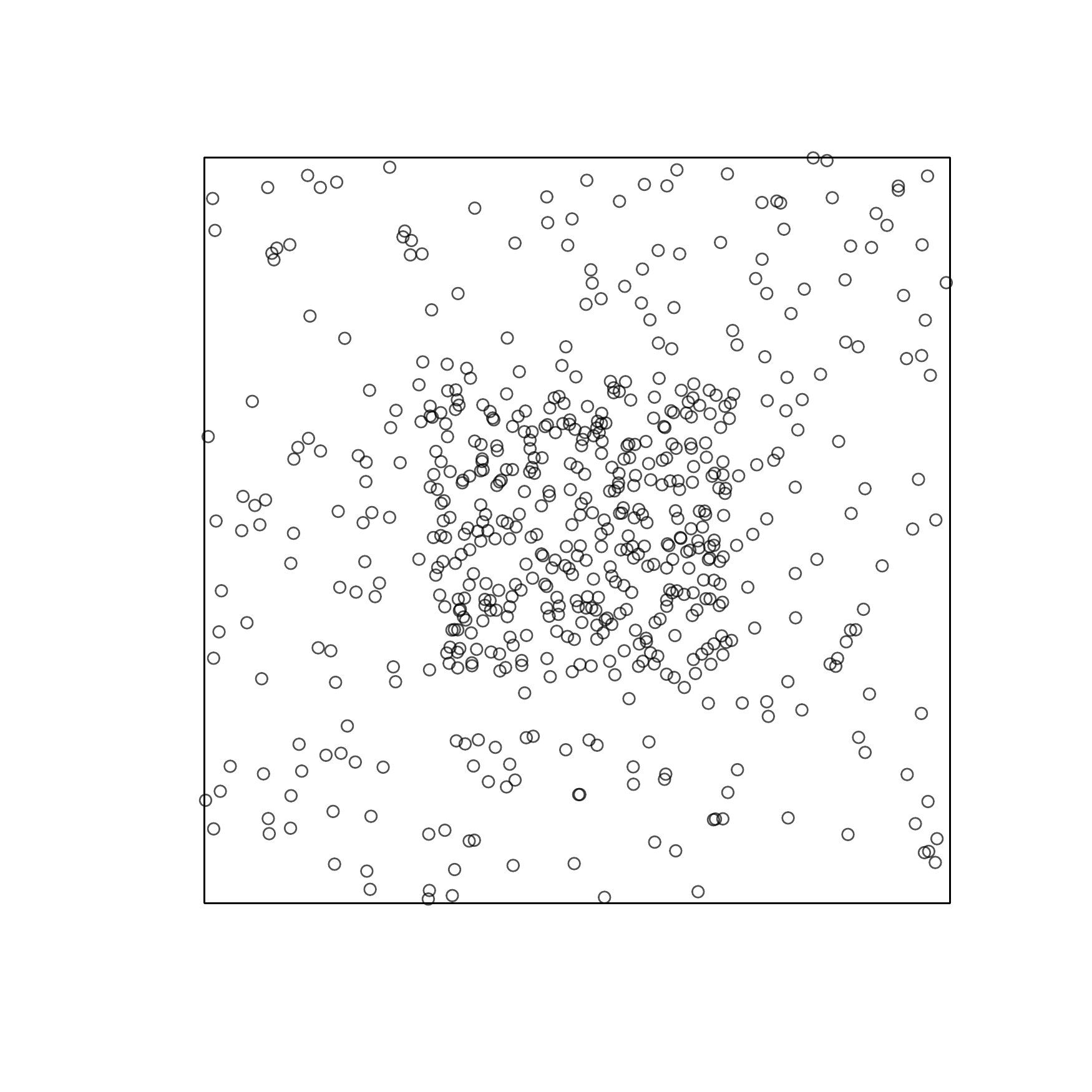}\hspace*{-1.73em}
  \includegraphics[scale=0.25]{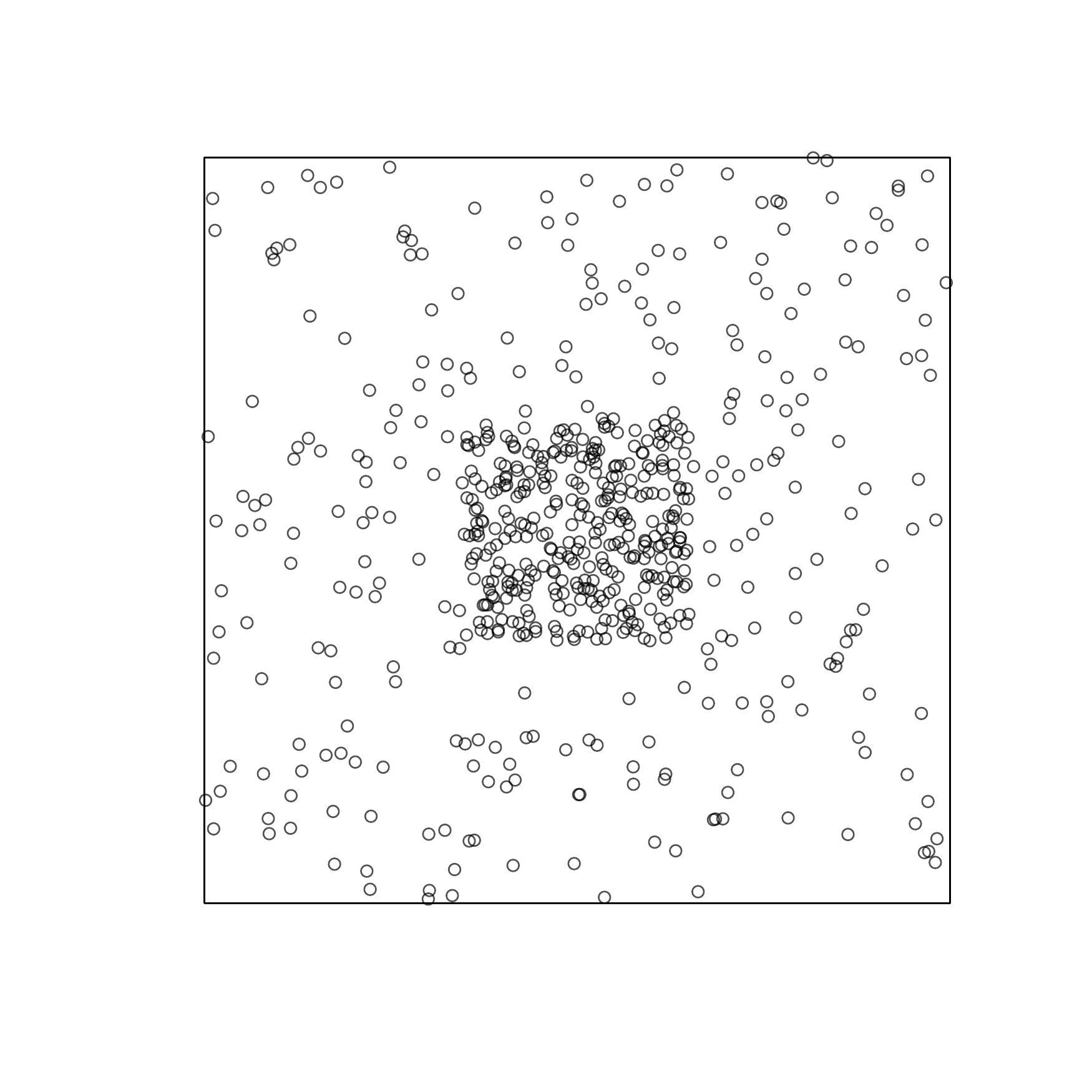}\vspace*{-1.99em}
  \\
  \includegraphics[scale=0.25]{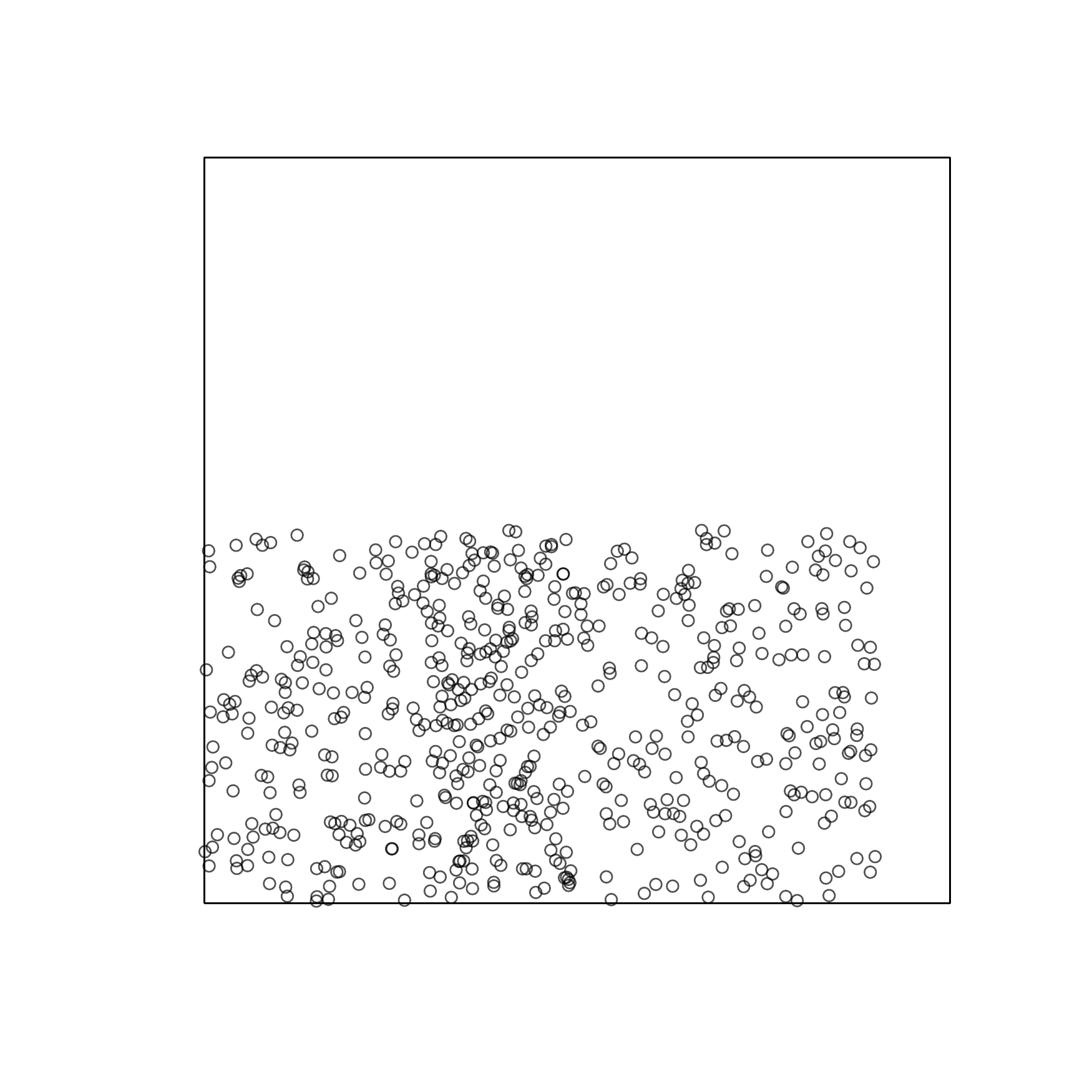}\hspace*{-1.73em}
  \includegraphics[scale=0.25]{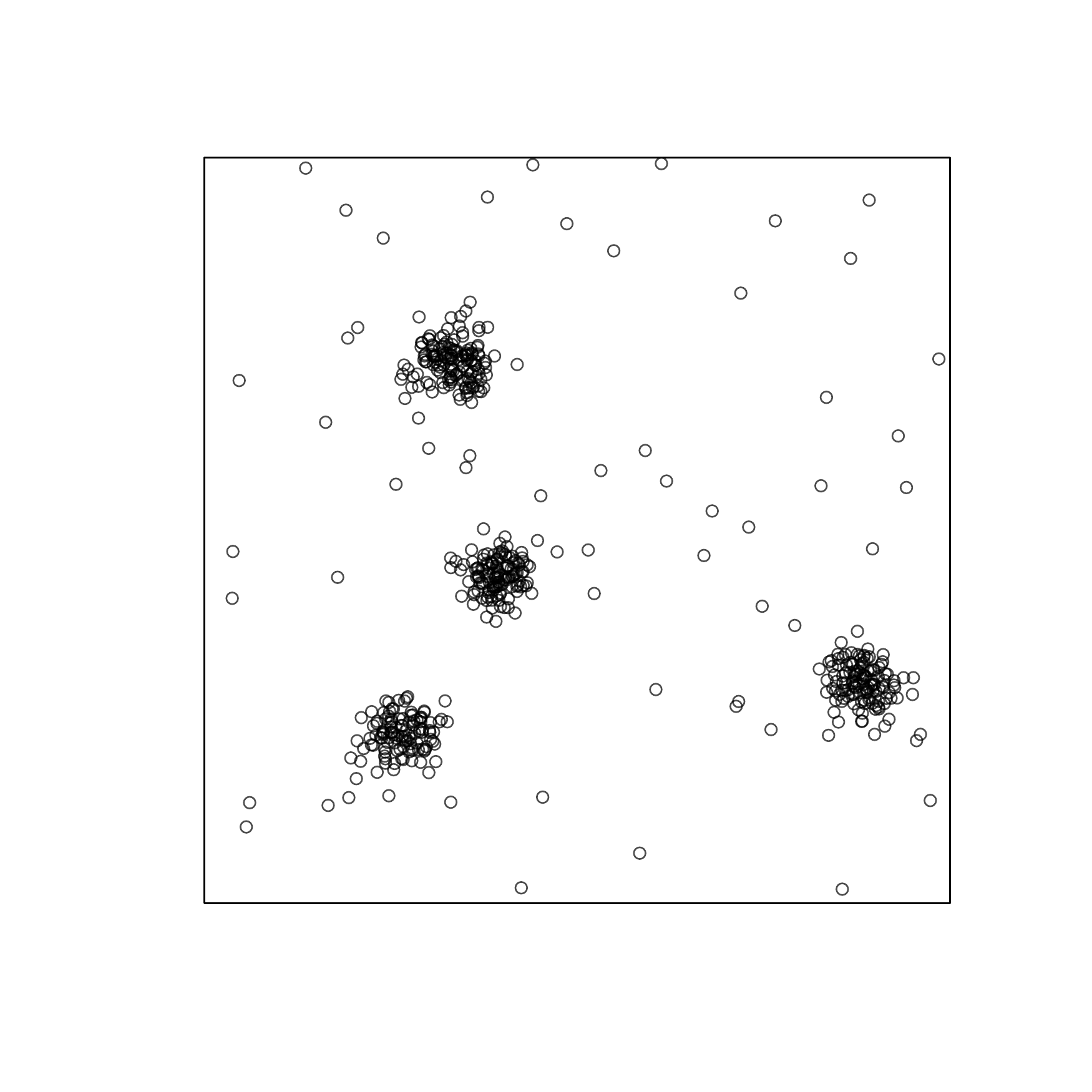}
\end{minipage}
We ran a simulation study with $100$ patterns from each scenario to compare the three algorithms. 
The results are in Table~\ref{tab:ImprovementResults}. 
As expected, all three algorithms find the exact solution every time. 
\\
In these calculations we set $\alpha=0.5$, so the cost of an empty barycenter is $C \cdot \frac{|\mathcal{A}|}{2}$. We have $d=\ell_2$ and $q=2$. 
For the computation we used the publicly available \textsf{R}-package \textsf{ttbary}, see~\cite{ttbary}.
\\
The runtime depends highly on the point pairs from which barycenter candidates are calculated. 
The larger the cutoff $C$, the more point pairs are taken into account and barycenter candidates have to be checked. 
Therefore the runtime gets higher with larger cutoff. 
For Algorithm~\ref{algo:Drezner} for the smallest cutoff $C=0.01$ the runtime is about $4-12$ seconds for the $100$ patterns combined. 
In Scenario $6$ the runtime for $C=0.01$ is about $81$ seconds for the $100$ runs, because even with the small cutoff due to the small clusters many barycenter candidates have to be calculated. For all the scenarios the runtime goes up to $1709-2574$ seconds for $C=0.3$. 
We also counted how many barycenter candidates had to be calculated by this algorithm in total for each scenario.
In Table~\ref{tab:ImprovementResults} we compare the runtime of the two improved Algorithms~\ref{algo:improvement} and~\ref{algo:improvement2} to the `original' runtime and compare how many barycenter candidates could be skipped by the improved algorithms.
\\
The column `skipped points' presents for each scenario and cutoff the \emph{relative} number of barycenter candidates that were skipped by this algorithm. 
\\
The values correspond to Algorithm~\ref{algo:improvement}/Algorithm~\ref{algo:improvement2}/Algorithm~\ref{algo:Drezner} (first and second improvement and original algorithm). 
For example in Scenario $1$, $C=0.01$ the $0.450/1.000/0$ means that Algorithm~\ref{algo:improvement} was able to skip $45\%$ of the barycenter candidates, Algorithm~\ref{algo:improvement2} skipped $100\%$ of the barycenter candidates, and of course Algorithm~\ref{algo:Drezner} skipped nothing.
\\
Similarly the column `time' presents for each scenario and cutoff the \emph{relative} time the algorithms took for the $100$ point patterns compared to the runtime of the original Algorithm~\ref{algo:Drezner}.
For example in Scenario $1$, $C=0.01$ the $0.702/0.252/1$ means that Algorithm~\ref{algo:improvement} was about $30\%$ faster and Algorithm~\ref{algo:improvement2} was about $75\%$ faster than Algorithm~\ref{algo:Drezner}.
\medskip

We can clearly see the connection between the amount of skipped barycenter candidates and the amount of time that is saved.
The first improvement, Algorithm~\ref{algo:improvement}, is almost never slower than the original algorithm and can for smaller cutoffs $C$ save up to $30\%$ of the runtime. 
The second improved version, Algorithm~\ref{algo:improvement2}, is much faster than the other two. 
The empty barycenter is in these scenarios for small cutoffs always the optimal solution. 
For cutoffs up to $C=0.1$ almost all point pairs can be skipped a priori. 
In scenario $1$ even for $C=0.2$ the runtime is below $1\%$ of the runtime of the original algorithm.

\begin{table}[!htb]
  \begin{footnotesize}
    \begin{tabular}{|l|c|c||c|c||c|c|}
      \hline
       & Scenario 1 & & Scenario 2 & & Scenario 3 &
      \\
      \hline
      $C=$ & skipped points & time &  skipped points & time & skipped points & time
      \\
      \hline
      $0.01$ & 0.450/1.000/0 & 0.702/0.252/1 & 0.383/1.000/0 & 0.703/0.144/1 & 0.289/1.000/0 & 0.778/0.133/1
      \\
      \hline
      $0.02$ & 0.273/1.000/0 & 0.773/0.082/1 & 0.128/1.000/0 & 0.911/0.044/1 & 0.105/1.000/0 & 0.911/0.038/1 
      \\
      \hline
      $0.03$ & 0.094/1.000/0 & 0.940/0.040/1 & 0.027/1.000/0 & 0.970/0.020/1 & 0.087/1.000/0 & 0.917/0.018/1 
      \\
      \hline
      $0.05$ & 0.006/1.000/0 & 0.999/0.015/1 & 0.004/1.000/0 & 0.992/0.008/1 & 0.081/1.000/0 & 0.922/0.007/1 
      \\
      \hline
      $0.075$ & 0.001/1.000/0 & 1.005/0.007/1 & 0.001/1.000/0 & 1.001/0.004/1 & 0.073/1.000/0 & 0.931/0.003/1 
      \\
      \hline
      $0.1$ & 0.000/1.000/0 & 1.005/0.004/1 & 0.000/1.000/0 & 1.001/0.002/1 & 0.064/1.000/0 & 0.938/0.002/1 
      \\
      \hline
      $0.2$ & 0.000/0.995/0 & 1.002/0.006/1 & 0.000/0.762/0 & 1.000/0.239/1 & 0.049/0.287/0 & 0.953/0.716/1 
      \\
      \hline
      $0.3$ & 0.000/0.263/0 & 1.000/0.738/1 & 0.000/0.009/0 & 1.001/0.991/1 & 0.018/0.041/0 & 0.983/0.960/1 
      \\
      \hline
      \hline
      & Scenario 4 & & Scenario 5 & & Scenario 6 & 
      \\
      \hline
      $C=$ & skipped points & time &  skipped points & time & skipped points & time
      \\
      \hline
      $0.01$ & 0.189/1.000/0 & 0.854/0.095/1 & 0.397/1.000/0 & 0.683/0.125/1 & 0.056/1.000/0 & 0.954/0.012/1
      \\
      \hline
      $0.02$ & 0.071/1.000/0 & 0.952/0.026/1 & 0.137/1.000/0 & 0.886/0.036/1 & 0.023/1.000/0 & 0.979/0.004/1 
      \\
      \hline
      $0.03$ & 0.070/1.000/0 & 0.945/0.012/1 & 0.044/1.000/0 & 0.963/0.017/1 & 0.012/1.000/0 & 0.990/0.002/1 
      \\
      \hline
      $0.05$ & 0.072/1.000/0 & 0.933/0.005/1 & 0.008/1.000/0 & 0.993/0.006/1 & 0.015/1.000/0 & 0.986/0.002/1 
      \\
      \hline
      $0.075$ & 0.076/1.000/0 & 0.926/0.002/1 & 0.004/1.000/0 & 0.999/0.003/1 & 0.080/0.987/0 & 0.920/0.014/1 
      \\
      \hline
      $0.1$ & 0.079/0.981/0 & 0.925/0.021/1 & 0.002/1.000/0 & 0.998/0.002/1 & 0.181/0.953/0 & 0.820/0.048/1 
      \\
      \hline
      $0.2$ & 0.116/0.231/0 & 0.887/0.772/1 & 0.000/0.171/0 & 1.000/0.831/1 & 0.127/0.528/0 & 0.875/0.473/1 
      \\
      \hline
      $0.3$ & 0.028/0.032/0 & 0.972/0.968/1 & 0.000/0.000/0 & 1.001/1.000/1 & 0.042/0.132/0 & 0.960/0.870/1 
      \\
      \hline
    \end{tabular}
    \caption{Comparison of the runtime of the two improved Algorithms~\ref{algo:improvement} and~\ref{algo:improvement2} to the `original' runtime of Algorithm~\ref{algo:Drezner} and how many barycenter candidates were skipped by the improved algorithms.
    The column `skipped points' presents for each scenario and cutoff the \emph{relative} number of barycenter candidates that were skipped by this algorithm. 
    The values correspond to Algorithm~\ref{algo:improvement}/Algorithm~\ref{algo:improvement2}/Algorithm~\ref{algo:Drezner} (first and second improvement and original algorithm). 
    Similarly the column `time' presents for each scenario and cutoff the \emph{relative} time the algorithms took for the $100$ point patterns compared to the runtime of the original Algorithm~\ref{algo:Drezner}.}
    \label{tab:ImprovementResults}
  \end{footnotesize}
\end{table}

\medskip
\subsection{Consequences for the barycenter algorithm of \cite{muller2020metrics}}\label{subsec:consequences}

As mentioned in Section~\ref{sec:applications} problem~\ref{eq:the_problem_prime}, that stems from \cite{muller2020metrics}, is identical to \eqref{empty} with $\alpha=\frac{1}{2}$. In the algorithm of \cite{muller2020metrics} this problem was solved by a fast heuristic:
\\
Starting with a point $x \in \mathbb{R}^k$ we calculate $\act(x)$, solve \Nam$(\act(x))$ with optimal solution $\xi^\ast$ and set $x \leftarrow \xi^\ast$.
The heuristic uses the idea that is proven in Lemma~\ref{lemma1}, that the optimal solution of \eqref{cutoff} must be an optimal solution of (\Nam)\ for some subset of $\mathcal{A}$. 
With this heuristic the objective function value cannot increase, since the distances to $\act(x)$ are optimized and the distances to $\con(x)$ cannot increase by definition of $\con(x)$.
\\

An implementation of the original algorithm of \cite{muller2020metrics} can be found in the publicly available R package \textsf{ttbary}, \cite{ttbary}. 
We implemented Algorithm~\ref{algo:improvement2} in the algorithm of \cite{muller2020metrics} to replace the heuristic.
In a simulation study we compared the implementation in \cite{ttbary} with our version in which the heuristic is replaced with Algorithm~\ref{algo:improvement2}. 

It turned out that doing the exact calculation instead of the heuristic for solving problem~\ref{eq:the_problem_prime} does not improve the algorithm of \cite{muller2020metrics} in general. 
In the algorithm the size of $\mathcal{A}$ in \eqref{empty} depends on the number of point patterns. 
The set $\mathcal{A}$ consists of exactly one point (including $\aleph$, see Section~\ref{sec:applications}) of every pattern. 
We compared the runtime and the resulting objective function value (cost) of the computed pseudo-barycenters.
For the three `groupsizes' of $20$, $50$ and $100$ point patterns per group we created $600$ groups each. 
Both algorithms had the same input for each of the $1800$ groups.
In our tests about half of the costs with the exact solutions of problem~\ref{eq:the_problem_prime} were smaller and half of the costs were larger compared to the heuristic. 
At the same time the runtime for the algorithm with the exact subroutine for \ref{eq:the_problem_prime} is about $3.5$, $13.5$ or $42$ times larger for the groupsizes of $20$, $50$ and $100$ respectively.
Since the heuristic is a lot faster and does not yield a worse solution we recommend to stay with the original version of the algorithm as it is presented in \cite{muller2020metrics}.

\section{Discussions and outlook}

In this paper we presented the problems \eqref{classic}, also known as the Weber-problem, and the extension \eqref{cutoff}, which is related to a problem studied by \cite{drezner1991facility}. Additionally we introduced the new barycenter problem \eqref{empty}, where we extend the classic problem by the option to have an empty solution.
In Sections~\ref{sec:cutoff} and \ref{sec:emptyCenter} we investigated under which conditions an optimal solution of \eqref{classic} is also an optimal solution to \eqref{cutoff} or \eqref{empty}.
We also investigated under which conditions an optimal solution to \eqref{classic} cannot be an optimal solution to \eqref{cutoff} or \eqref{empty}.
Most results are based solely on the geometric structure of the dataset, like the diameter of the set or the mean pairwise distance between its points.
The summaries of the results can be found in Table~\ref{table:summary1} and the statements thereafter and at the end of Section~\ref{sec:emptyCenter}.
\medskip

For the average problem we typically do not know if we can reduce \eqref{cutoff} or \eqref{empty} to \eqref{classic}. We presented two improvements of the algorithm introduced by \cite{drezner1991facility} to solve \eqref{cutoff} and \eqref{empty} more efficiently.
We furthermore gave an algorithm for solving \eqref{cutoff} simultaneously for \emph{all} $C \geq 0$ by solving $\mathcal{O}(n)$ problems of type \eqref{cutoff} for specified values of $C$.
\medskip

For future research it might be interesting to generalize the improved algorithms to the original problem stated by \cite{drezner1991facility}, who allowed different cutoffs for every point.

Another interesting topic is to find new criteria to determine beforehand if solving \eqref{classic} is sufficient. 
Another algorithmic idea is to split the original problem into subproblems that can be solved independently, where one optimal solution of the subproblems is guaranteed to be the optimal solution of the original problem. 
One could also study how \eqref{cutoff} simplifies for special cases like the $\ell_1$-metric, where we can optimize separately over the $k$ dimensions.
These findings could help to solve the problems \eqref{cutoff} and \eqref{empty} faster in the future.

\newpage

\bibliography{location}
\bibliographystyle{alpha}

\end{document}